\documentclass{siamart190516}
\pdfoutput=1  

\usepackage{amsfonts}
\usepackage{graphicx}
\usepackage{epstopdf}
\usepackage{algorithmic}
\usepackage{amsopn}

\usepackage{physics}
\usepackage{mathtools}
\usepackage{enumitem}
\usepackage{multicol}

\ifpdf
\DeclareGraphicsExtensions{.eps,.pdf,.png,.jpg}
\else
\DeclareGraphicsExtensions{.eps}
\fi

\DeclareMathOperator*{\FFT}{FFT}
\DeclareMathOperator*{\IFFT}{IFFT}
\DeclareMathOperator*{\sinc}{sinc}

\newsiamremark{remark}{Remark}
\newsiamremark{hypothesis}{Hypothesis}
\crefname{hypothesis}{Hypothesis}{Hypotheses}
\newsiamthm{claim}{Claim}
\newsiamthm{prop}{Proposition}

\graphicspath{{./graphics}}

\headers{Fourier-Informed Knot Placement Schemes}{D. Lenz, O. Marin, V. Mahadevan,  R. Yeh, and T. Peterka}

\title{Fourier-Informed Knot Placement Schemes for B-Spline Approximation%
\thanks{Submitted to the editors DATE.
		\funding{This work is supported by the
		U.S. Department of Energy, Office of Science, Advanced Scientific Computing Research under Contract DE-AC02-06CH11357, and the Exascale
		Computing Project (Contract No. 17-SC-20-SC), a collaborative effort of the U.S. Department of Energy Office of Science and the National Nuclear Security Administration.}}
}

\author{David Lenz\thanks{Argonne National Laboratory, Lemont, IL 60439, USA
			(\email{dlenz@anl.gov})}
	\and Oana Marin\thanks{Argonne National Laboratory, Lemont, IL 60439, USA
			(\email{oanam@mcs.anl.gov})}
	\and Vijay Mahadevan\thanks{Argonne National Laboratory, Lemont, IL 60439, USA
			(\email{mahadevan@anl.gov})}
	\and Raine Yeh\thanks{Purdue University, West Lafayette, IN 47907, USA
			(\email{yeh10@purdue.edu})}
	\and Tom Peterka\thanks{Argonne National Laboratory, Lemont, IL 60439, USA
			(\email{tpeterka@mcs.anl.gov})}
}

\begin{document}
\maketitle

\begin{abstract}
	Fitting B-splines to discrete data is especially challenging when the given data contain noise, jumps, or corners. Here, we describe how periodic data sets with these features can be efficiently and robustly approximated with B-splines by analyzing the Fourier spectrum of the data. Our method uses a collection of spectral filters to produce different indicator functions that guide effective knot placement. In particular, we describe how spectral filters can be used to compute high-order derivatives, smoothed versions of noisy data, and the locations of jump discontinuities. Our knot placement method can combine one or more of these indicators to place knots that align with the qualitative features of the data, leading to accurate B-spline approximations without needing many knots. The method we introduce is direct and does not require any intermediate B-spline fitting before choosing the final knot vector. Aside from a fast Fourier transform to transfer to and from Fourier space, the method runs in linear time with very little communication. The method is applied to several test cases in one and two dimensions, including data sets with jump discontinuities and noise. These tests show that the method can fit discontinuous data without spurious oscillations and remains accurate in the presence of noise.
\end{abstract}



\section{Introduction}
B-spline curves are a class of continuous functions used extensively in the approximation of discrete data sets, particularly in geometric modeling~\cite{lin2018geometric} and isogeometric analysis (IGA)~\cite{hughes2005iga}.  These functions possess several important properties making them useful for approximating geometric curves and surfaces, and leading to their use by cutting-edge engineering~\cite{rhino3d} and visualization~\cite{glvis-tool} software. In recent years, B-splines (and their generalization, non-uniform rational basis splines, or NURBS) have been used to approximate and represent large-scale scientific data sets, which may not be easily recognized as surfaces, solids, or the like. A recent advance in this direction is the application of B-spline functional approximation to model data produced by scientific simulations in high-performance computing (HPC) settings~\cite{peterka_ldav18}. This approach has been extended to NURBS curves as well~\cite{nashed2019}.

The approximation of scientific data sets by splines requires consideration of several details sometimes de-emphasized in the computer-aided design (CAD) and visualization literature. For instance, spline representations of scientific data must be accurate enough for subsequent data analysis and postprocessing; this level of accuracy may be higher than what is required by geometric applications. Furthermore, data sets may be high-dimensional, contain complex combinations of smooth and non-smooth behavior, or be represented on scattered data. A common problem when modeling large, complex data sets is the lack of any prior knowledge regarding discontinuities, interfaces, or other qualitative features of the data. In order to efficiently approximate scientific data sets with exceedingly low error, B-spline approximation schemes must account for distinctive features within the data automatically.

B-spline curves are piecewise-polynomial $C^k$ functions (for some prescribed $k \geq 0$) by definition; as a result, the accuracy of a B-spline approximation depends on the location of its piecewise-polynomial breakpoints, or ``knots.'' However, optimal knot placement is a challenging problem. The location of spline knots must be specified before the model fitting begins, but choosing the optimal set of locations amounts to a nonlinear minimization problem with many local extrema~\cite{jupp1978free}. In practice, two categories of numerical schemes have been employed to choose knot locations: iterative methods and direct methods. Broadly speaking, iterative methods repeatedly fit B-splines on different knot distributions until certain criteria are satsified. These methods can guarantee error bounds, but do not have a bounded time complexity. For instance, an iterative scheme based on interval bisection was proposed by Liang et al.~\cite{liang2017}. Another approach was taken by Galvez et al.~\cite{galvez2015}, where genetic algorithms were applied to choose optimal knots locations. On the other hand, direct (or ``one-shot'') methods employ heuristics to choose knot locations and typically run faster than iterative schemes. Direct methods can be used to choose knots in a standalone manner, or to produce an initial guess for knot locations that serves as input to an iterative scheme.

In this paper, we propose a direct knot placement scheme that uses the Fourier spectrum of the input data to make intelligent knot placement choices. From the Fourier coefficients, we can accurately compute high-order derivatives, the locations of jump discontinuities, and (when the input is corrupted by noise) a smoothed version of the input data. Depending on the application, one or more of these indicators can be used to choose a knot vector for the B-spline approximation. For example, given a data set which is known to be smooth, high-order derivatives can be used to carry out the method of Yeh et al.~\cite{yeh2020knot}. Furthermore, using Fourier coefficients to compute the locations of discontinuities allows us to handle data that is only piecewise-smooth, which was not straightforward in the original formulation by Yeh et al. Crucially, the locations where the data breaks into smooth pieces need not be known ahead of time. Additionally, when a data set is corrupted by noise, a smoothed proxy for the input can be inferred directly from the Fourier spectrum and need not be computed explicitly. Then, a knot vector for the noisy data can be chosen based on the derivatives of the smooth proxy. Since certain operations (such as convolution, shifting, etc.) are notably simplified in Fourier space in contrast with their real space counterparts, complex preprocessing steps become natural in Fourier space.  The overall time complexity for this method is $O(N\,\log N)$; however, the entirety of the method apart from a fixed number of fast Fourier transforms runs in linear time.

A number of other notable heuristics have been developed to guide direct knot placement methods. Li et al.~\cite{li2004} used discrete curvature and angular deflection as a way to choose knot locations. In their approach, the authors computed an integrated discrete curvature function and placed knots such that each sequential pair of knots covered equal areas of discrete curvature. They also applied a smoothing procedure to compute the curvature of noisy data. This approach was recently extended from curvature to high-order derivatives by Yeh et al.~\cite{yeh2020knot}, whereby the $q^{th}$ derivative is integrated and used to place knots ($q$ being the order of the approximating spline). They showed that derivatives higher than second order (i.e., curvature) provide a robust heuristic when higher-order splines are considered. However, techniques for handling noisy data were not discussed. Michel and Zidna~\cite{michel2020} also recently considered the use of high-order derivatives to choose knot locations. Their method blends multiple heuristics including curvature, angular deflection, and derivatives into one composite function. The full algorithm considers 231 combinations of the various heuristics and generates a knot distribution for each combination. After fitting a B-spline with each knot distribution, the algorithm returns the approximation with minimal $L^2$ error.

The recent works of Yeh et al.~\cite{yeh2020knot} and Michel and Zidna~\cite{michel2020} illustrate the power of using high-order derivatives as a knot placement heuristic. The method of Michel and Zidna is robust, fitting data with jumps and corners. However, this procedure is relatively slow compared with that of Yeh et al., which performs efficiently on data sets larger than those considered by Michel and Zidna. In the context of modeling data for scientific applications, an ideal scheme would be robust enough to account for features that typically arise in scientific data sets (like noise, corners, and jumps) while still running efficiently enough to execute within high-performance computing workflows. 

Throughout this paper, we also make a point to emphasize how a Fourier-informed approach to B-spline fitting allows multiple indicators to be used in tandem to produce better approximations. For instance, combining the techniques presented here, we can intelligently choose a knot vector to fit a data set with jumps and spikes that is also corrupted by noise. At the mathematical level, each operation that we present can be realized as a filter applied to the Fourier spectrum of the input. As a result, the computation of various indicators follow the same pattern: first, compute the Fourier transform of the input; second, apply some filter to the Fourier coefficients (via a pointwise vector product); third, take the inverse Fourier transform of the filtered spectrum. Computing derivatives, jump locations, and smooth proxies all follow this workflow; obtaining the desired indicator simply amounts to choosing the right filter. 

The remainder of the paper is structured as follows. In \cref{sec:math-background}, we present basic mathematical methods and definitions that will be used in the remainder of the paper. This includes an error analysis of derivative-informed knot placement not presented before. \Cref{sec:method} describes the specifics of the proposed method in detail. Within \cref{sec:method}, each subsection defines a different filtering scheme and the scenarios in which it may be applied. In \cref{sec:numerical}, we present the results of numerical experiments that validate the behavior of the methods presented in \cref{sec:method} and compare with an existing method. Extensions to our method are discussed in \cref{sec:future}, and conclusions are presented in \cref{sec:conclusion}.

\section{Mathematical Preliminaries}\label{sec:math-background}
At the core of every B-spline approximation is a knot vector defining how the piecewise-polynomial function transitions from one piece to the next. Splines constructed from a knot vector incorporating the qualitative features of the data are generally more accurate than those based on uniform knot vectors. As such, it is important to understand how spline approximation accuracy depends on the knot vector and how this dependence may be exploited to construct near-optimal knot vectors. 

To begin, we set out some basic definitions of B-splines and knots in \cref{ssec:spline-rep}. In \cref{ssec:deriv-method}, theoretical results on the approximatin power of splines are applied to provide a new justification for the method of Yeh et al.~\cite{yeh2020knot} that was not included in their paper. Fundamental definitions and operations from Fourier theory are then described in \cref{ssec:fourier-theory}.

\subsection{B-spline Representation}\label{ssec:spline-rep}
We introduce basic definitions and properties for the one-dimensional case, but higher-dimensional extensions follow using tensor product splines. As a result, each of the fundemental definitions in higher dimensions can be described in terms of their one-dimensional analogues. A more in-depth exposition of tensor product splines may be found in the work of Habermann and Kindermann~\cite{habermann2007multi}.

A one-dimensional B-spline curve in $\mathbb{R}^d$ with $n$ control points is a parameterized curve
\begin{equation}\label{eq:spline-def}
\mathbf{C}(u) = \sum_{j=0}^{n-1} N_{j,p}(u) \mathbf{P_j},
\end{equation}
where each $N_{j,p}$ is a piecewise-polynomial function of degree $p$ and each $\mathbf{P_j} \in \mathbb{R}^d$ is a ``control point'' in $d$-dimensional space. We say that $\mathbf{C}$ is a B-spline curve of ``degree $p$'' or (equivalently) ``order $p+1$.'' The order of a B-spline is defined to be $q = p+1$, and this parameter is frequently used in the numerical analysis of splines. 

The functions $N_{j,p}$ are the B-spline basis functions and are defined on the parameter space $[0,1] \subset \mathbb{R}$. This parameter space is divided into subdomains by a nondecreasing sequence of knots $k_0,k_1,\ldots,k_{n+p} \in [0,1]$; the precise location of each $k_j$ is left unspecified for the moment---choosing the location of each $k_j$ is the subject of this paper. We emphasize that when the number of control points, $n$, is specified, the number of knots is always $n+p+1$ by definition. Given a knot vector $\mathbf{k} = \{k_j\}$ and a polynomial degree $p$, the B-spline basis functions are defined by the recurrence:
\begin{equation}
\begin{aligned}
N_{j,0}(u) &= \begin{dcases}
1 & \text{ if } u \in [k_j, k_{j+1})\\
0 & \text{ otherwise }
\end{dcases}\\
N_{j,p+1}(u) &= \frac{u-u_j}{u_{j+p} - u_j}N_{j,p}(u) + \frac{u_{j+p+1}-u}{u_{j+p+1}-u_{j+1}}N_{j+1,p}(u), \qquad p \geq 0,
\end{aligned}
\end{equation}
which was introduced by Cox~\cite{cox1972numerical} and de Boor~\cite{de1972calculating}.

The B-spline fitting problem seeks to represent a discrete data set with a function of the form in \cref{eq:spline-def}. Given a set of discrete data $\{\mathbf{Q_0},\ldots,\mathbf{Q_{m-1}}\} \subset \mathbb{R}^d$ representing a one-dimensional curve, we are interested in minimizing the quantities $|\mathbf{C}(u_i) - \mathbf{Q_i}|$, $0 \leq i \leq m-1$, where each $u_i \in [0,1]$ is some real parameter chosen to correspond to the point $\mathbf{Q_i}$. The construction of a spline approximant begins by choosing the knot vector $\mathbf{k} = \{k_j\}$ and the polynomial order $q$. Then, the locations of the control points $\mathbf{P_j}$ that minimize the $L^2$ error $E(\mathbf{k},q) = \left(\sum_i |\mathbf{C}(u_i) - \mathbf{Q_i}|^2\right)^{1/2}$ are computed. This computation is straightforward, requiring the solution of a linear least squares system of size $m \times n$, where $m$ is the number of input points and $n$ the number of control points.

In this paper, we consider scalar-valued input data sampled on a regular grid in one or two dimensions. Given $m$ uniform samples on an interval $[a,b]$, we denote the sampled function by $f$ and the $i^{th}$ sampled value by $f_i$. The parameters associated to each input point are likewise assumed to be uniformly distributed; hence, $u_i = i/(m-1)$ for $0 \leq i \leq m-1$.
Similar conventions are made for scalar data sets sampled on a regular two-dimensional grid; the only changes are in notation. For 2D data sampled on the rectangle $[a_1,b_1]\times [a_2,b_2]$, we denote by $m_1$ and $m_2$ the number of inputs in each dimension; the total input size is therefore $m_1m_2$. Input values $f_{i_1,i_2}$ denote samples of a function $f$ from the $i_1^{th}$ row and $i_2^{th}$ column. As in the case of 1D data, we assume uniformly distributed parameters $u_1, u_2$ in each dimension.

Our primary error metrics considered when assessing the accuracy of a B-spline approximation are the root mean-squared (RMS) and maximum error:
\begin{equation}\label{eq:rms-error}
e_{RMS} = \sqrt{\frac{1}{m}\sum_{i=0}^{m-1} \abs*{C(u_i) - f_i}^2 } \qquad
	e_{max} = \max_{0 \leq i \leq m-1} \abs*{C(u_i) - f_i}.
\end{equation}
In higher dimensions, we apply analogous RMS and maximum error norms where the indexing for the sum or max operator is expressed over all dimensions.

\subsection{Error Analysis of Derivative-Informed Knot Placement}\label{ssec:deriv-method}
A direct method using high-order derivatives to guide the placement of knots was recently introduced by Yeh et al.~\cite{yeh2020knot}. The same reference shows that this method compares favorably against a number of popular knot placement methods and is computationally efficient. We summarize the basic structure of this method here and introduce a new justification for the accuracy of the method as well.

Let $f^{(q)}(x)$ be the $q^{th}$ derivative of $f$.  We define the ``feature function'' of $f$ to be the function on the parameter space $[0,1]$ given by
\begin{equation}\label{eq:feat-func}
F(u) = \left|f^{(q)}\left(a + u(b-a)\right)\right|^{\frac{1}{q}}.
\end{equation}
The feature curve $F(u)$ is a measure of the $q^{th}$ derivative of $f$ at a point corresponding to the parameter $u$. This may be considered a high-order analogue to curvature, which is based on the second derivative of a signal. The method is designed to subdivide the parameter space with knots such that each knot span contains roughly the same amount of integrated feature. To do this, we compute a normalized cumulative distribution function (CDF) of $F(u)$ and choose interior knots $\{k_j\}_{j=q}^{n-1}$ such that every knot span is mapped to an interval of equal size. 

To be precise, for $u \in [0,1]$ let 
\begin{equation}
G(u) = \dfrac{\int_0^u F(t)\,dt}{\int_0^1 F(t)\,dt}
\end{equation} be the CDF of the feature curve. We will refer to $G(u)$ as the ``Feature CDF.''  The full knot vector $\mathbf{k} = \{k_j\}_{j=0}^{n+q-1}$ is defined by
\begin{equation}\label{eq:knot-placement}
k_j = \begin{dcases}
0 & \text{ if } 0 \leq j \leq q-1,\\
1 & \text{ if } n \leq j \leq n+q-1,\\
G^{-1}\left(\frac{j-q+1}{n-q+1}\right) & \text{ otherwise.}
\end{dcases}
\end{equation}
Under this definition, the CDF function $G$ satisfies the property
\begin{equation}\label{eq:cdf-spacing}
G(k_{j+1}) - G(k_j) = \frac{1}{n-q+1} \quad \text{for } q-1 \leq j \leq n.
\end{equation}
Furthermore, $\mathbf{k}$ is constructed with $q$ repeated knots at each endpoint, a convention in B-spline approximation that specifies the behavior of a spline at its endpoints. Finally, we remark that while $G$ is assumed here to be invertible, this is not a requirement in practice.\footnote{
	It is possible that $G$ is constant on a subinterval of $[0,1]$ (for instance if $f$ itself is locally constant). In this case, $G$ would not be invertible and the knot construction procedure would be ambiguous. In practice, we consider a perturbed CDF, $G_\epsilon(u) = G(u) + \epsilon u$,
	where $\epsilon$ is small; this perturbed CDF is always monotonically increasing. The choice of $\epsilon$ is arbitrary and we use $\epsilon = (1000m)^{-1}$ in our implementation. For the sake of clarity, in the remainder of this paper we shall  assume that $G(u)$ is strictly increasing, understanding that it can always be perturbed to be so.
}

Yeh et al.~\cite{yeh2020knot} introduced this knot placement scheme based on empirical considerations.  One strength that was observed is that the method generates a spline with roughly equal approximation error throughout the domain, no matter the distance between two adjacent knots. For example, consider the signal shown in \cref{fig:knot-place-effect}. Here, uniform knot placement is compared with knots chosen according to \cref{eq:knot-placement} (which we call ``derivative-informed'' knot placement). In the residual plot on the right, we observe that uniform knot placement is much less accurate near the center of the domain, but more accurate at the extremes of the domain. In contrast, the derivative-informed knots produce an approximation with roughly equal accuracy throughout the domain, even though the knot spans vary in size. In effect, derivative-informed knot placement ``distributes'' error equally throughout the domain. In the remainder of this subsection, we provide a theoretical explanation for this behavior.

\begin{figure}[h!]
	\centering
	\begin{minipage}{.49\textwidth}
	\includegraphics[width=\textwidth]{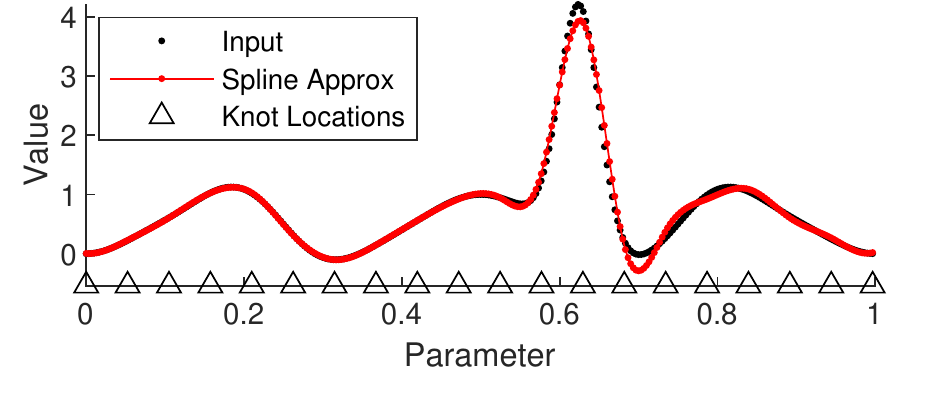}\\
	\includegraphics[width=\textwidth]{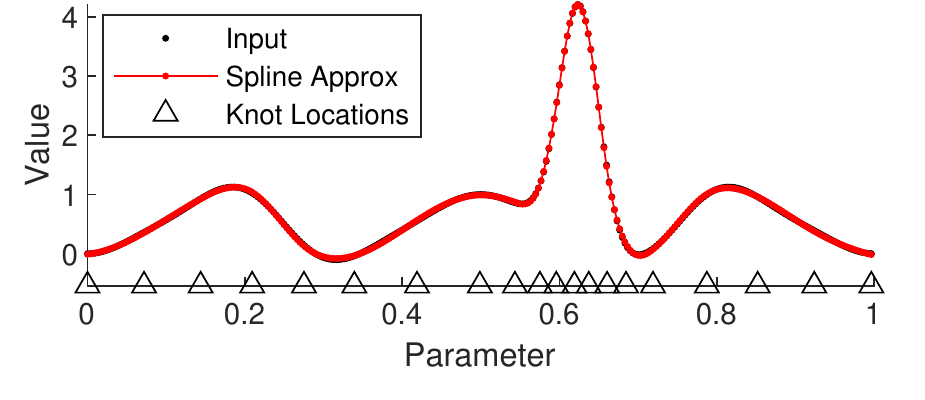} 
	\end{minipage}
	\begin{minipage}{.49\textwidth}
	\includegraphics[width=\textwidth]{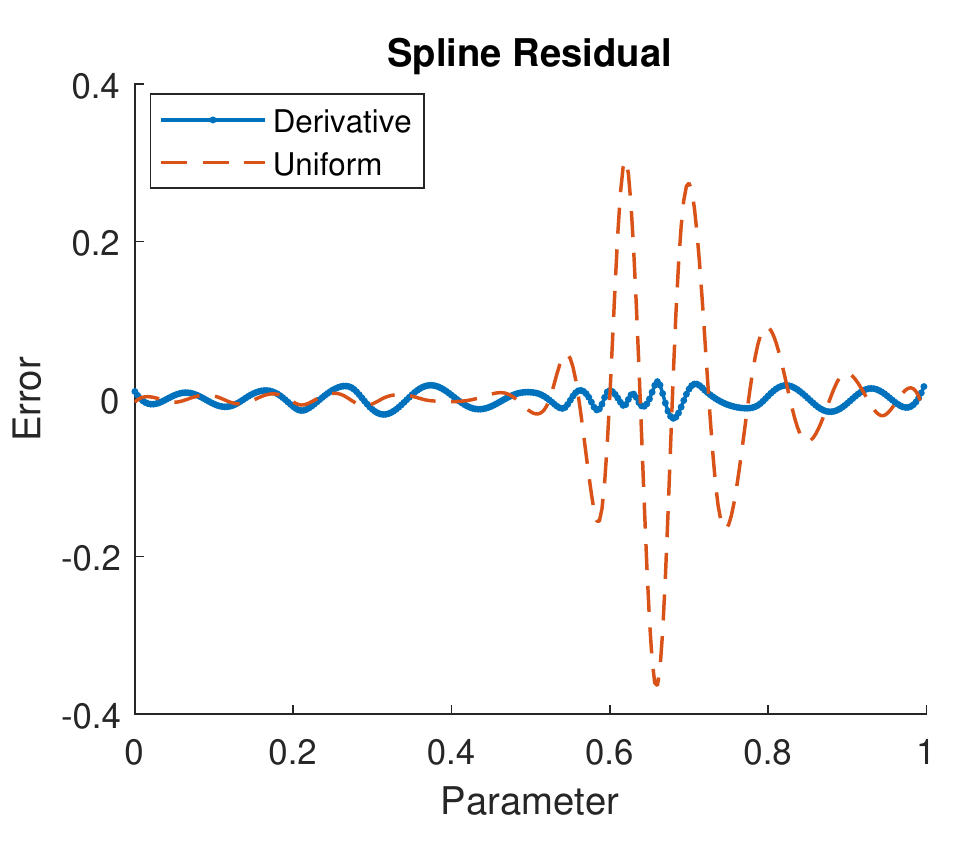}
	\end{minipage}
	\caption{The effect of knot placement on the distribution of error in a spline approximation.Uniform knot placement (top left) and derivative-informed knot placement (bottom left) are applied to data with a steep peak. The residuals for each approximation are compared at right.}
	\label{fig:knot-place-effect}
\end{figure}

For every pair of consecutive knots $k_j, k_{j+1}$ and $q-1\leq j\leq n-1$, we define the $j^{th}$ knot span $I_j = [k_j,k_{j+1}]$. The size of this span is denoted $h_j = k_{j+1}-k_j$. For every knot, we are also interested in the largest knot span within a small neighborhood of the knot. To describe the neighborhood of interest, we set $\widetilde{I}_j = [k_{j-q+1}, k_{j+q}]$ and $\widetilde{h}_j = \max(h_{j-q+1},\ldots,h_{j+q-1})$.

The central fact to be established is the following.
\begin{claim}\label{claim:deriv-error}
Let $f$ be a function with $q+1$ continuous derivatives and $\mathbf{k} = \{k_j\}$ the knot vector defined according to \cref{eq:knot-placement}. If the knot span widths $\{h_j\}$ are locally quasi-uniform and $C$ is the B-spline of order $q$ with knots $\mathbf{k}$ that best approximates $f$, then 
\begin{equation}
\norm{f-C}_{C(I_j)} \leq A\frac{\int_0^1 \abs{f^{(q)}(t)}^{1/q}\,dt}{n-q+1}
\end{equation}
for all $q-1 \leq j \leq n-1$ and some constant $A>0$.
\end{claim}

When $\mathbf{k}$ is chosen according to \cref{eq:knot-placement}, we have shown that $G(k_{j+1}) - G(k_j) = (n-q+1)^{-1}$ is constant (see \cref{eq:cdf-spacing}). The definition of $G$ implies that 
\begin{equation}\label{eq:deriv-1}
\frac{1}{n-q+1} = \frac{\int_0^{k_{j+1}} F(t)\,dt}{\int_0^1 F(t)\,dt} - \frac{\int_0^{k_j} F(t)\,dt}{\int_0^1 F(t)\,dt} = \frac{\int_{k_j}^{k_{j+1}} F(t)\,dt}{\int_0^1 F(t)\,dt} \approx h_j \frac{\norm{F}_{C(I_j)}}{\int_0^1 F(t)\,dt}
\end{equation}
for all $q-1 \leq j \leq n-1$. 

Next, it is useful to recast the right-hand side of \cref{eq:deriv-1} in terms of the derivative of $f$. By the definition of the feature function $F$, 
\begin{equation}
	\norm{F}_{C(I_j)} = \norm{\abs{f^{(q)}}^{1/q}}_{C(I_j)} = \norm{f^{(q)}}_{C(I_j)}^{1/q},
\end{equation} 
and consequently
\begin{equation}
	h_j^q \norm{f^{(q)}}_{C(I_j)} \approx \frac{\int_0^1 F(t)\,dt}{(n-q+1)^q}
\end{equation}
for all $q-1 \leq j \leq n-1$.

Next, we apply a fundamental property of spline approximation.
\begin{lemma}
	Let $f$ be a function with $q$ continuous derivatives, $\mathbf{k}$ a knot vector, and $C$ the best-fit B-spline of order $q$. Then for each knot span $I_j$, $C$ satisfies
\begin{equation}\label{eq:sp-estimate}
	\|f - C\|_{C(I_j)} \leq A_1 \widetilde{h}_j^q \|f^{(q)}\|_{C(\widetilde{I}_j)} \quad \text{for } q-1 \leq j \leq n-1,
\end{equation}
for some constant $A_1>0$.
\end{lemma}
\begin{proof} 
	See \cite[Thm. 6.24]{basic_theory}.
\end{proof}

To apply this lemma, two assumptions are needed. If we assume that $f$ has at least $q+1$ continuous derivatives, then there is another constant $A_2>0$ such that $\norm{f^{(q)}}_{C(\widetilde{I}_j)} \leq A_2 \norm{f^{(q)}}_{C(I_j)}$. While this is a strong assumption, our intent at present is only to study the approximation in regions where $f$ is smooth. Treating regions where $f$ is not smooth is the motivation behind the methods in \cref{ssec:smooth-proxy} and \cref{ssec:jump-detect}. Additionally, if we assume that the knot span widths $\{h_j\}$ are locally quasi-uniform, then there is yet another constant $A_3>0$ such that $\widetilde{h}_j \leq A_3 h_j$ for $q-1\leq j\leq n-1$.

With these two assumptions, and setting $A=A_1A_2A_3>0$, we deduce from \cref{eq:sp-estimate} that
\begin{equation}
\begin{aligned}
	\|f - C\|_{C(I_j)} &\leq A_1 \widetilde{h}_j^q \|f^{(q)}\|_{C(\widetilde{I}_j)}\\
	&\leq A h_j^q \norm{f^{(q)}}_{C(I_j)}\\
	&\approx A \frac{\int_0^1 F(t)\,dt}{(n-q+1)^q}
\end{aligned}
\end{equation}
for all $q-1 \leq j \leq n-1$. This establishes a bound for the maximum error on each knot span that, crucially, does not depend on the width of the span. Thus, as the number of control points $n$ increases, the error within each knot span is decreased simultaneously.

\subsection{Spectral Representation of Signals} \label{ssec:fourier-theory}
Let $\mathcal{F}$ denote the Fourier transform operator and $\mathcal{F}^{-1}$ the inverse Fourier transform. Given a continuous function $f(x)$, we denote the Fourier transform of $f$ as $\hat{f}(\xi) := \mathcal{F}(f)(\xi)$. We say that $f$ is a function in physical space while $\hat{f}$ is a function in Fourier space or spectral space. 

Owing to the convolution theorem, which states that convolutions in real space become products in Fourier space, we can regard any function $K(\xi)$ in frequency space as a filter. Taking the inverse Fourier transform after filtering $\hat{f}$ can produce useful information about $f$. For instance, when $K(\xi) = i\xi$, the Fourier duality between multiplication and differentiation tells us that $\mathcal{F}^{-1}(K(\xi)\hat{f}(\xi)) = f'(x)$. Other filters may be applied using this same technique to produce smoothed versions of $f$ (c.f. \cref{ssec:smooth-proxy}) or indicator functions locating jump discontinuities in $f$ (c.f. \cref{ssec:jump-detect}). 
Specifically, given a function $f(x)$ and spectral filter $K(\xi)$, we are interested in the filtered function 
\begin{equation}\label{eq:filter}
\widetilde{f}(x) = \mathcal{F}^{-1}(K(\xi)\hat{f}(\xi)).
\end{equation}
Different properties of $f$ can be deduced by analyzing the result of applying different spectral filters, which is a cornerstone of modern signal processing.

The fast Fourier transform (FFT) decomposes a discrete signal into a combination of frequency components. Given a discrete signal $\{f_k\}_0^{N-1}$, the FFT computes a set of Fourier coefficients $\{F_n\}_0^{N-1}$. The FFT and its inverse (IFFT) are defined by
\begin{equation}
\begin{aligned}
F_n &= \sum_{k=0}^{N-1} f_k e^{-2i \pi k n / N},  \quad 0 \leq n \leq N-1,\\
f_k &= \frac{1}{N} \sum_{n=0}^{N-1} F_n e^{2i \pi k n / N}, \quad 0 \leq k \leq N-1.
\end{aligned}
\end{equation}
Spectral filtering techniques can be applied to discrete data in an analogous way. Given a discrete filter $K(\xi_n)$, an input signal $\{f_k\}_0^{N-1}$, and its discrete Fourier coefficients $\{F_n\}_{0}^{N-1}$, we compute the filtered signal according to the equation
\begin{equation}\label{eq:discrete-filter}
\widetilde{f}_k = \IFFT(K(\xi_n)\odot F_n),
\end{equation}
where $\odot$ denotes pointwise multiplication (Hadamard product) and $\{\xi_n\}_0^{N-1}$ is the vector of frequencies corresponding to each Fourier coefficient.

It may be noted that filters can be applied in physical space with a convolution instead of the method above.
While these two characterizations are equivalent, computational concerns lead us to compute in spectral space. When $\IFFT(K)$ has global support (which is the case for some filters considered herein), the computational complexity of the convolution needed to filter in physical space is quadratic. It will also be necessary to apply multiple filters to the same function, which further amplifies this difference in complexity. In this case, only one FFT/IFFT pair is required, no matter how many filters are chained together.

\section{Fourier-Informed Methods for Knot Placement} \label{sec:method}
Since a B-spline is a highly smooth curve, approximating non-smooth data poses a challenge unless certain steps are taken. For example, B-spline approximations of data with jump discontinuities often produce splines with overshoots and undershoots near the jump in a manner similar to the Gibbs phenomenon~\cite{richards1991gibbs}. However, this can be addressed by co-locating several knots near the jump, or (if the precise location of the jump is known) adding a knot of high multiplicity at the jump location.
In a similar way, data sets with sharp peaks or corners are typically smoothed out by B-spline approximations, unless several knots are placed in a neighborhood around the peak.

Our method incorporates three indicators to inform the placement of knots. These indicators are:
\begin{enumerate}
	\item High-order derivatives of the input
	\item Locations of jump discontinuities in the data and its first derivative
	\item A smooth proxy function
\end{enumerate}
Each indicator affects the knot vector in a different way and all are computed directly from the Fourier transform of the input data. 

\subsection{Knot Placement with High-Order Derivatives}\label{ssec:deriv-knots}
If $f(x)$ is a function with $q$ continuous derivatives, then the $q^{th}$ derivative $f^{(q)}(x)$ satisfies the identity
\begin{equation}
f^{(q)}(x) = \mathcal{F}^{-1}((i\xi)^q \hat{f}(\xi)).
\end{equation}
In order to compute derivatives of a discrete function $f_k$, we define the spectral filters
\begin{equation}\label{eq:deriv-filter}
	K_{d,q}(\xi) = (i\xi)^q
\end{equation}
for any non-negative integer $q$. When $f_k$ is sampled from a smoothly periodic function, the filtered signal
\begin{equation}\label{eq:discrete-deriv-approx}
	\widetilde{f}_k = \IFFT(K_{d,q}(\xi_n)\odot \FFT(f_k))
\end{equation}
will be an accurate approximation of $f^{(q)}$ at the $k^{th}$ grid point.

The accuracy of high-order derivatives computed with \cref{eq:discrete-deriv-approx} is observed to be several orders of magnitude higher than standard differentiation techniques when the data is smooth and periodic. For comparison, we considered two other common approaches to numerical differentiation: finite differences and estimation from a moving least-squares approximation~\cite{movinglsq}. \Cref{fig:smooth-deriv-accuracy} compares these three methods for a signal with 500 points. The finite difference scheme used was a first-order central difference. To estimate the derivative of order $q$ at a given point with a moving least-squares approximation, a polynomial of degree $q$ was fit to a patch of $k$ points around the given point. The derivative approximation is the $q^{th}$ derivative of the best-fit polynomial at the point. In our implementation, we set $k=\max(q+1, 2q-1)$.

From \cref{fig:smooth-deriv-accuracy} it is apparent that derivative computation by spectral filtering produces a highly accurate approximation. However, we remark that this approximation suffers for nonsmooth and/or nonperiodic data. Certain kinds of nonsmooth signals can be handled with the smoothing filter described in the following subsection, but nonperiodic signals are in general not suited to approximation by this method. There are techniques that may extend this approach to nonperiodic signals (c.f Future Work, \cref{sec:future}); at present, however, we consider only periodic data in our analysis.

\begin{figure}[h]
	\centering
	\includegraphics[width=.4\textwidth]{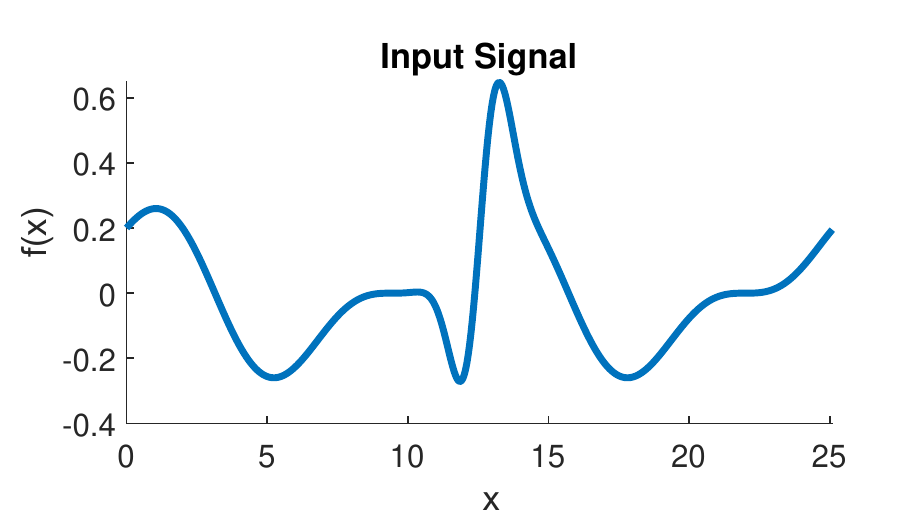} \qquad
	\includegraphics[width=.4\textwidth]{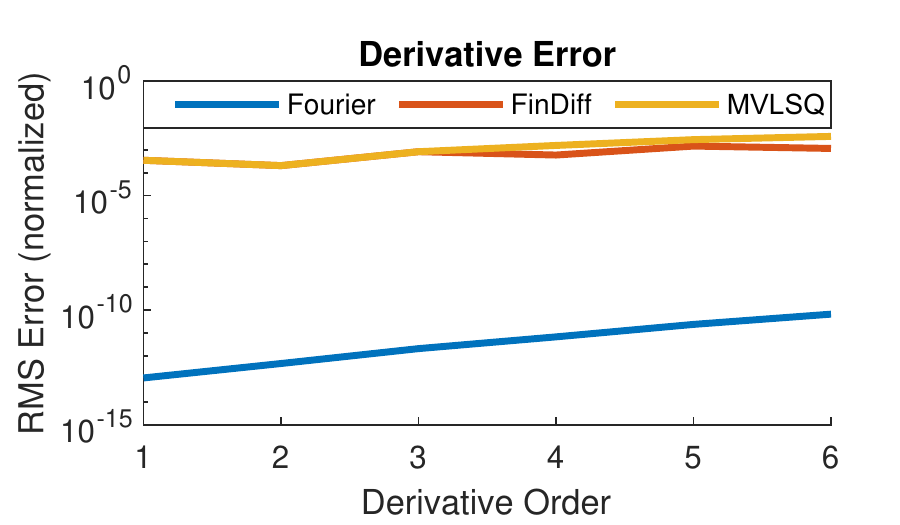}
	\caption{Left: A sample smooth periodic function. Right: Accuracy of numerical differentiation for the first six derivatives, using three methods. ``Fourier'' is the method described by \cref{eq:discrete-deriv-approx}, ``FinDiff'' is a central differencing scheme, and ``MVLSQ'' is an estimate based on a moving least-squares approximation.}
	\label{fig:smooth-deriv-accuracy}
\end{figure}

The representation in \cref{eq:discrete-deriv-approx} also indicates that once Fourier coefficients are computed for the first time, the computation of derivatives incurs only costs associated with a Hadamard product. The computational complexity of the FFT is the well-known $\mathcal{O}(N \, \log N)$. From the perspective of building a B-spline data model, this is a fixed cost resolved by dedicated software for efficient FFT implementations. Furthermore, the pointwise multiplications in Fourier space are embarrassingly parallel, regardless of the physical dimension or derivative order. The computational complexity of derivatives in Fourier space is $\mathcal{O}(N)$ and no communication is needed. This can be significantly faster than even vector inner products, which require a reduction that can bottleneck the computation on a parallel distributed data set. This aspect marks the superiority of approaching numerical differentiation in Fourier space, as opposed to real space, where the width of the stencil dictates the number of floating point operations, memory accesses, and potentially communication costs (in the case of parallel distributed data).

For smooth signals and numerical derivatives computed by \cref{eq:discrete-deriv-approx}, knot placement can be carried out via the method described in \cref{ssec:deriv-method}. By computing  derivatives in spectral space, the feature function will be constructed from highly-accurate derivatives in an efficient and scalable way.

\subsection{Knot Vectors for Noisy Data}\label{ssec:smooth-proxy}
While the use of derivatives as a knot placement heuristic has numerous advantages, one of its biggest disadvantages is sensitivity to noise. Whether computed by finite differences or Fourier transforms, high-order numerical derivatives of functions perturbed by noise can be highly inaccurate. One way to limit this numerical instability is to approximate noisy input data with a smooth proxy function. Conti et al.~\cite{conti2001cubic} took this approach in their method of using third derivatives to fit cubic splines. However, their construction relied heavily on a smoothing parameter that needs to be chosen to match the strength of the noise. 

The problem of smoothing a noisy signal is fundamental to the field of signal processing. One of the most common methods of smoothing a discrete signal is to convolve the signal with a smoothing kernel. By the convolution theorem, an equivalent operation can be carried out in Fourier space. When derivatives are computed using the method of \cref{ssec:deriv-knots}, the Fourier coefficients need not be recomputed, saving one $O(N\log N)$ operation.

There are numerous ways to filter a noisy signal and undertaking a thorough comparison of these methods is beyond the scope of this paper. Instead, we illustrate the concept using a simple low-pass Gaussian filter and show that we achieve good results with this straightforward approach. Suppose a given data set is uniformly sampled with grid size $h$, and a Gaussian blur is applied via the kernel 
\begin{equation}
L(x) = \frac{1}{\sigma \sqrt{2 \pi}} e^{-\frac{x^2}{2\sigma^2}},
\end{equation}
which is the standardized mean-zero Gaussian with standard deviation $\sigma$. If the standard deviation is taken to be one-half the grid spacing, $\sigma = h/2$. Then 
\begin{equation}
L(x) = \frac{1}{2h \sqrt{2\pi}} e^{-\frac{2x^2}{h^2}}
\end{equation}
and the smoothed proxy is computed via the discrete convolution of $L$ with the data.

The corresponding smoothing filter in Fourier space is given by the Fourier transform of $L(x)$. 
Thus, the formula for this filter is
\begin{equation}\label{eq:smooth-filter}
K_s(\xi) = \hat{L}(\xi) = e^{-\frac{\pi^2 h^2}{2}\xi^2}.
\end{equation}
In effect, filtering the Fourier spectrum with $K_s$ reduces the strength of high-frequency components within a noisy signal.  This produces behavior similar to applying a low-pass filter. \Cref{fig:noisy-spectrum} illustrates the effect of $K_s$ on a noisy data set and its Fourier spectrum.

\begin{figure}[h]
	\centering
	\includegraphics[width=.32\textwidth]{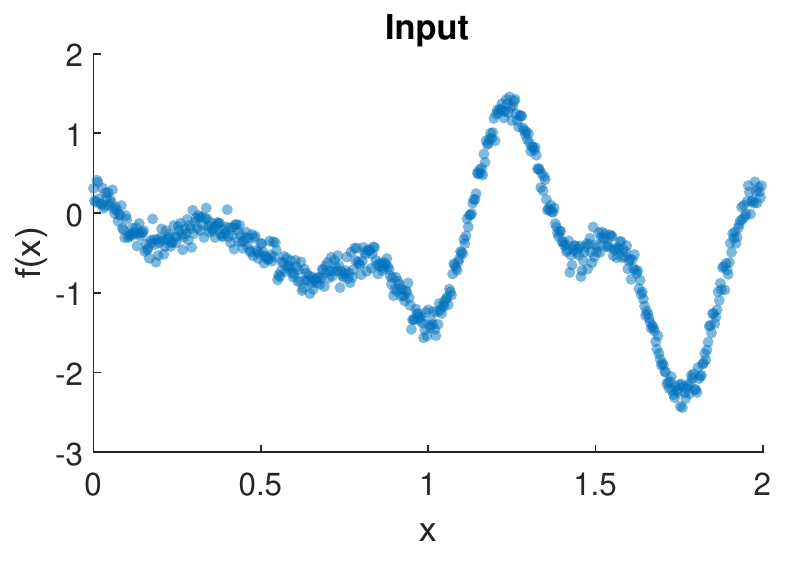} 
	\includegraphics[width=.32\textwidth]{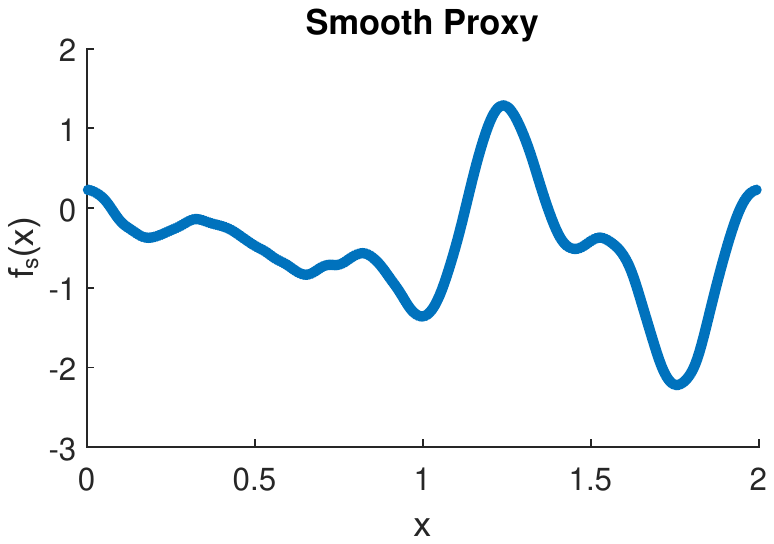} 
	\includegraphics[width=.32\textwidth]{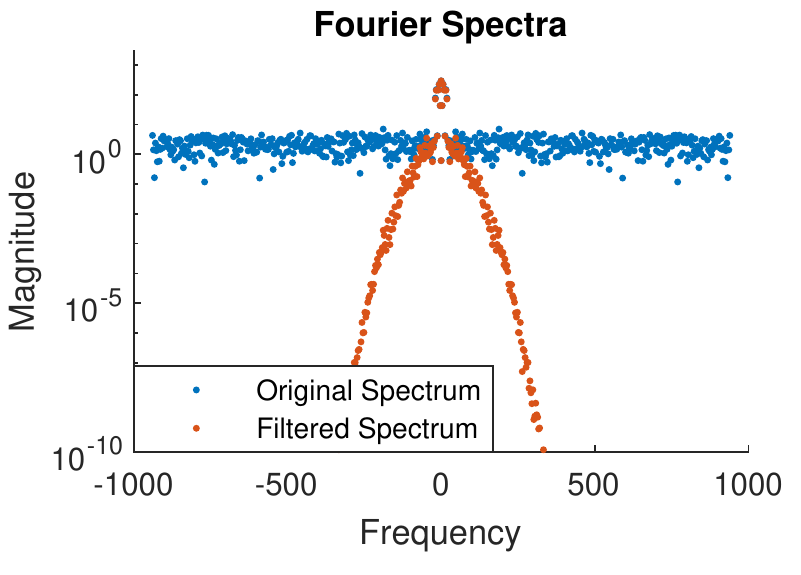}
	
	\vspace{-3ex}
	\caption{Left: A noisy signal $f(x)$. Center: Smooth proxy function $f_s(x)$ produced by spectral filtering. Right: Original Fourier spectrum (blue) and filtered spectrum (orange); note the filtered spectrum decays while the original does not.} 
	\label{fig:noisy-spectrum}
\end{figure}

In practice, the smooth proxy function (at right in the figure) is never computed explicitly; it is included for illustrative purposes only. By smoothing and computing derivatives in Fourier space, we can approximate the derivatives of the proxy without ever computing the smooth proxy in physical space. Computing the derivatives of the proxy function is carried out via
\begin{equation}
	\widetilde{f}_k = \IFFT(K_s(\xi_n) \odot K_q(\xi_n) \odot \FFT(f_k)).
\end{equation} 
Thus, the entire process of smoothing a noisy signal for the purposes of derivative computation amounts to applying a second filter to the existing workflow of \cref{ssec:deriv-knots}. As discussed above, this is an embarrasingly parallel $O(N)$ operation. In addition, the smoothing operation requires no parameter tuning, which allows for noisy data to be handled in an automatic way.

\subsection{Knot Vectors for Discontinuous Data}\label{ssec:jump-detect}
In the presence of discontinuities, B-spline approximations with equispaced knots will exhibit undesirable overshoots and undershoots. Our present focus is the treatment of functions with jump discontinuities in the value and/or in the derivative. To distinguish these two cases, we will refer to jumps in the function value as ``$C^0$ jumps'' and jumps in the derivative of the function as ``$C^1$ jumps.'' In general, spurious oscillations can be reduced by clustering several knots near a discontinuity; see \cref{fig:jump-example}. Therefore, it is highly desirable to identify any discontinuities in the data and treat them accordingly during the construction of the knot vector.

\begin{figure}[h]
	\centering
	\includegraphics[width=.47\textwidth]{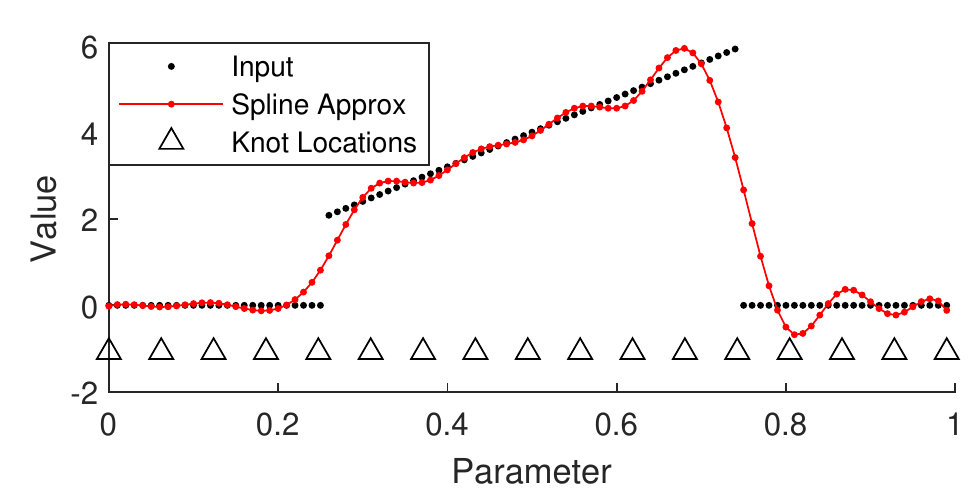} \hspace{.02\textwidth}
	\includegraphics[width=.47\textwidth]{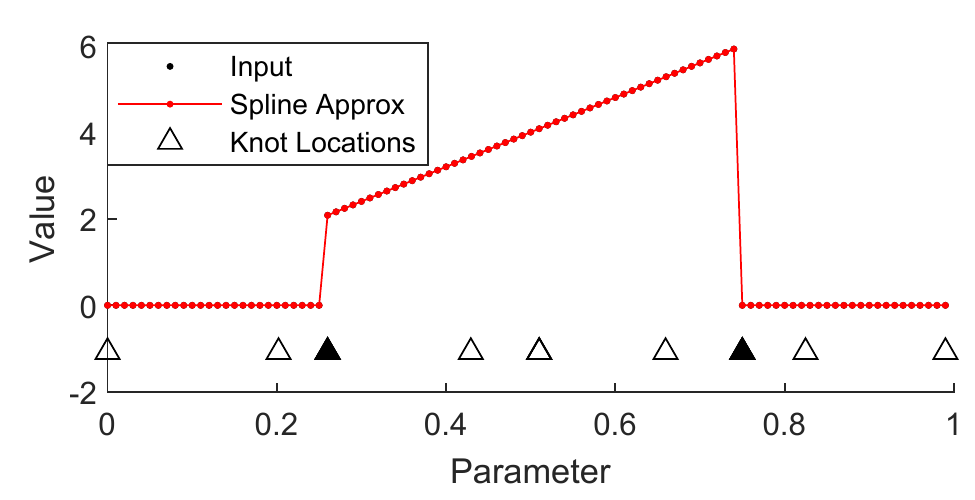}
	\caption{Effect of uniform knot placement (left) and jump-aware knot placement (right) on a cubic spline approximation of data with $C^0$ jumps. Filled-in triangles denote high-multiplicity knots.}
	\label{fig:jump-example}
\end{figure} 

In this subsection we apply a spectral filter that identifies $C^0$ and $C^1$ jumps directly from a signal's Fourier coefficients. Filters of this type were introduced by Tadmor, Gelb, Cates, and others~\cite{ cates2007detecting, gelb1999detection, gelb2000detection, gelb2008detection} in the context of edge detection. For an overview of this subject we refer the reader to the expository work of Tadmor~\cite{tadmor2007filters}; for a more detailed treatment of jump detection in the first derivative, see Cates and Gelb~\cite{gelb2008detection}.

We utilize the exponential concentration factor of Gelb and Tadmor~\cite{gelb2000detection},
\begin{equation}\label{eq:exp-factor}
\sigma_{exp}(\xi) = \frac{2 \pi i}{c_\alpha} \xi e^{\frac{1}{\alpha \xi(\xi-1)}},
\end{equation}
where $\alpha$ is a parameter and $c_\alpha$ is a normalization constant dependent on $\alpha$. Following the implementation of Gelb and Tadmor~\cite{gelb2000detection}, we set $\alpha = 6$; the corresponding normalization is $c_6 \approx 0.34$. The spectral filter associated to this concentration kernel, which we will use to locate $C^0$ and $C^1$ jumps, is
\begin{equation}\label{eq:exp-filter}
K_j(\xi) := \text{sign}(\xi)\sigma_{exp}\left(\frac{2\xi}{m}\right) \sinc\left(\frac{\pi \xi}{m}\right),
\end{equation}
where $m$ is the length of the input signal.

By filtering the Fourier coefficients with $K_j$ and then inverting the FFT, we produce a detector that isolates $C^0$ and $C^1$ jumps within the input signal:
\begin{equation}\label{eq:jump-detector}
J(x_k) := \IFFT( K_j(\xi_n) \odot \FFT(f_k)) =  
\begin{dcases}
	O(1) & \text{if there is a } C^0 \text{ jump at } x_k\\
	O(1/m) & \text{if there is a } C^1 \text{ jump at } x_k\\
	O(e^{-\gamma m}) & \text{otherwise}
\end{dcases}.
\end{equation}
Here $0 < \gamma \leq 1$ is a parameter describing the exponential decay rate away from jumps.

The precise magnitude of $J(x)$ depends on the size of the jump as well as its type ($C^0$ or $C^1$). In particular, if $f(x)$ has a $C^0$ jump of size $l$ at $x = c$, then $J(c) = O(l)$. If $f(x)$ has a $C^1$ jump of size $l$, then $J(c) = O(l/m)$. Therefore, the location of jump discontinuities can be determined automatically by identifying the local extrema of $J(x)$ above a user-specified threshold. By setting the value of $l$, the user has control over what they consider to be significant. It is also worth noting that the parameter $l$ does not enter into the definition of the filter at all (c.f. \crefrange{eq:exp-factor}{eq:exp-filter}); it is only used when searching for jumps by interpreting the content of $J(x)$.

\Cref{fig:jump-detect-ex} illustrates how this jump detector works in practice. Given an input that is smoothly periodic except for two $C^0$ jumps, the function $J(x)$ produces two spikes at those discontinuities. In addition, the height of these spikes is directly related to the size of their corresponding jumps. It is the correspondence between jump height in $f(x)$ and spike height in $J(x)$ that allows one to scan only for discontinuities of a certain size. 

\begin{figure}
	\centering
	\includegraphics[width=.45\textwidth]{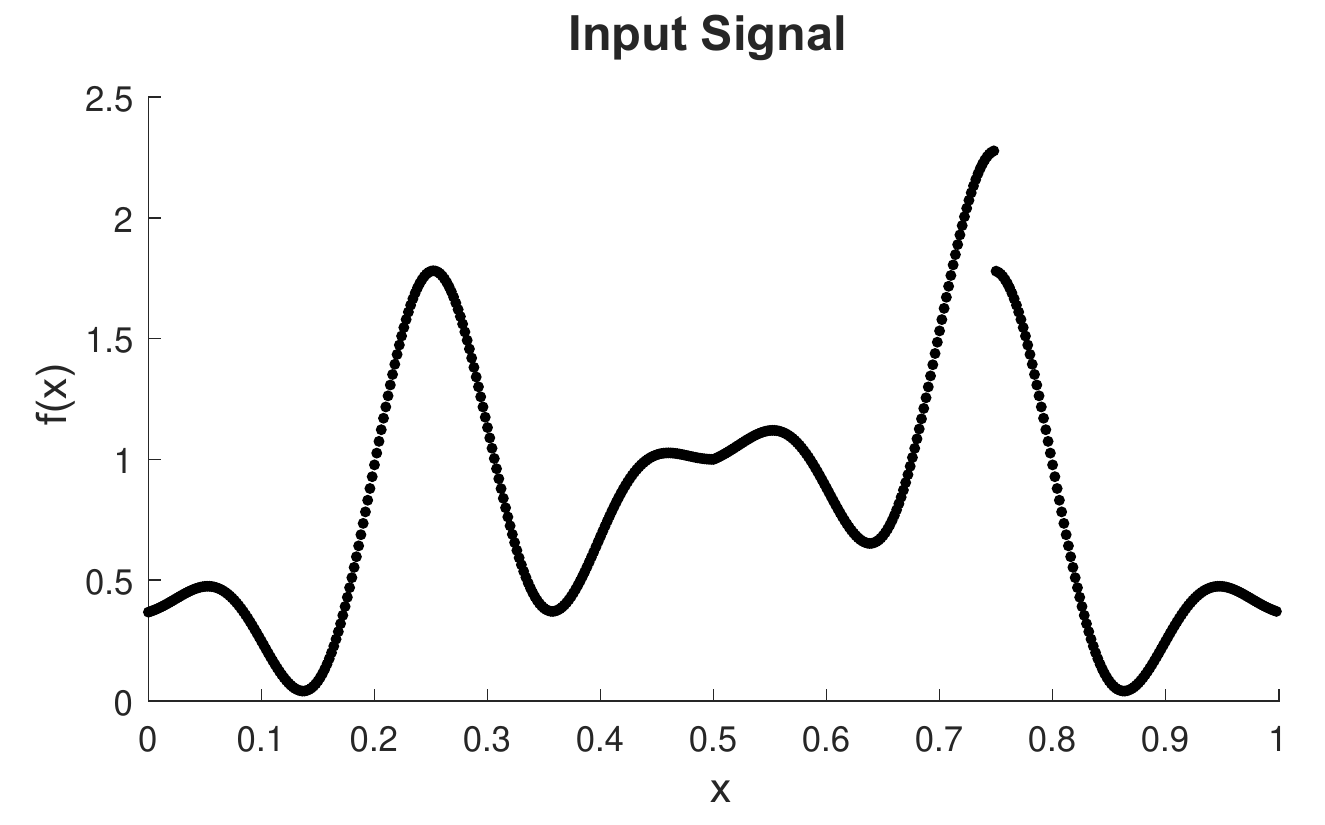} \hspace{.02\textwidth}
	\includegraphics[width=.45\textwidth]{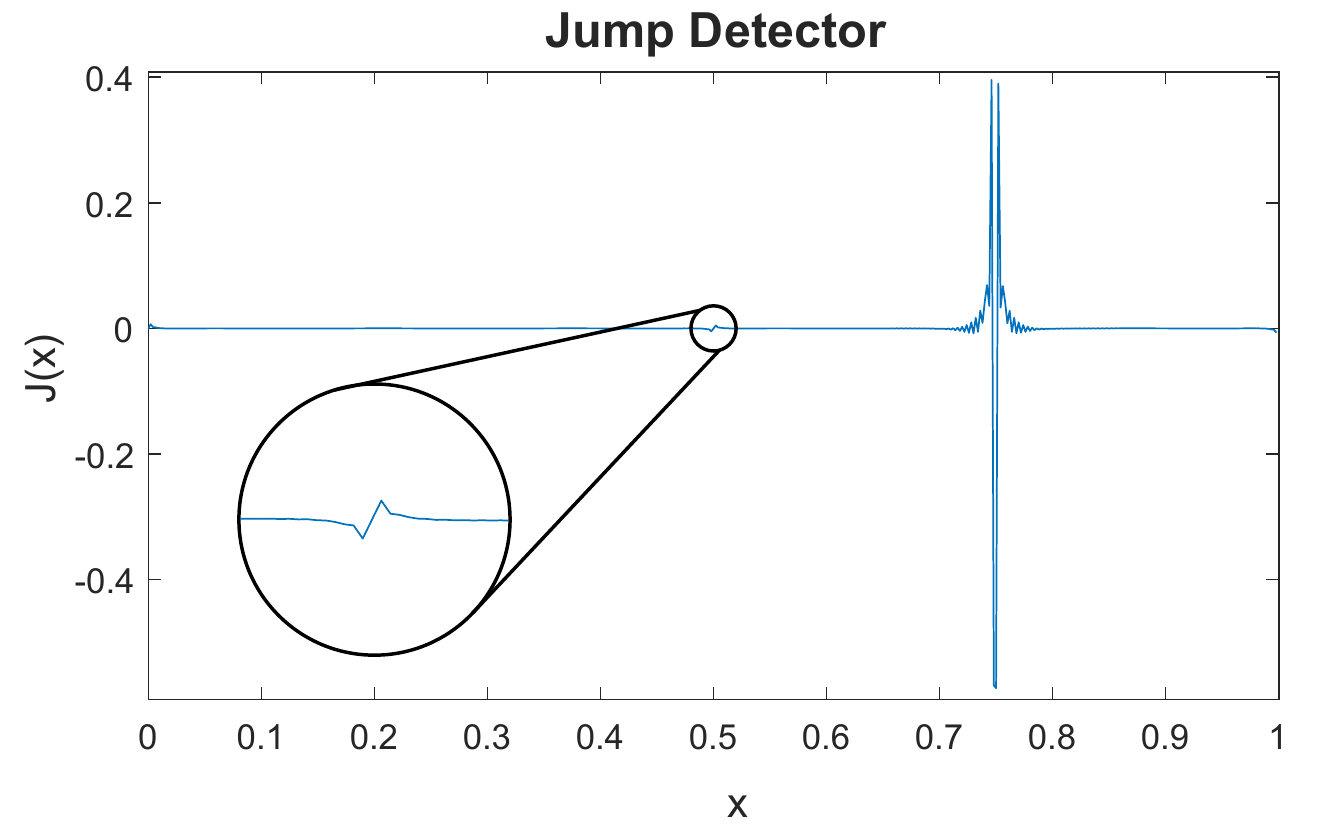}
	\caption{Left: Piecewise smooth periodic function with one $C^1$ jump at $x=0.5$ and one $C^0$ jump at $x=0.75$. Right: Output of the jump detector $J(x)$ for the input signal. Note that the size of the jumps in the input corresponds to the height of the spikes in the indicator.}
	\label{fig:jump-detect-ex}	
\end{figure}

In order to identify $C^1$ jumps in a signal, we adopt the two-pass approach of Cates and Gelb~\cite{cates2007detecting}. After $C^0$ jumps are identified, the entire jump indicator function may be scaled by a factor of $m$ and searched again for spikes where are $O(1)$. By \cref{eq:jump-detector}, this ensures that $C^1$ jumps will have the correct magnitude. The only extra step involved is to differentiate $C^1$ jumps from $C^0$ jumps. By scaling $J(x)$ by a factor of $m$, any spikes corresponding to $C^0$ jumps will also be picked up by the second pass searching for $C^1$ jumps. To determine the true type of a jump, we simply ignore any candidates for a $C^1$ jump that are detected in a neighborhood around a previously detected $C^0$ jump.

Once jumps are located within the input, knots can then be placed to account for the nonsmooth behavior. A discontinuous B-spline may be constructed by inserting a knot of multiplicity $q$ precisely at the detected jump. Under this construction, the resulting spline will exhibit no oscillations. In effect, the process that fits control points to the left and right of the discontinuity is decoupled; that is, the behavior of the signal on one side of the high-multiplicity knot does not affect the fit on the other side. Therefore, in our method we place a knot of multiplicity $q$ at the peak of every spike in $J(x)$ marking a $C^0$ jump. In order to handle $C^1$ jumps, we place a knot of multiplicity $q-1$ at the peak of each spike marking a $C^1$ jump.

After placing knots at each jump discontinuity, knots are placed throughout the smooth regions using the method in \cref{ssec:deriv-knots}. Since the data are nonsmooth (at each jump), it is necessary to construct a smooth proxy function for the input and compute the derivatives of the proxy. This step can be done immediately using the filter described in \cref{ssec:smooth-proxy}. We emphasize that adding this smoothing step requires essentially no extra work, since the Fourier coefficients were computed during the jump detection step. Smoothing the data simply requires two filters of the Fourier coefficients instead of one.

\subsection{Higher-dimensional data}\label{ssec:2d-method}
The methods described in the three previous subsections are all computed with a combination of FFTs and pointwise vector products. As a result, extensions to two dimensions and higher are straightforward and efficient. For simplicity, we will describe the two-dimensional case in detail. However, applications to higher dimensions follow the same structure as the two-dimensional case.

On an equispaced tensor product grid in two dimensions, the Fourier spectrum of a function $f: \mathbb{R}^2 \to \mathbb{R}$ may be computed with the 2D FFT. In this case, the Fourier coefficients $\hat{f}_{ij}$ lie on a 2D grid. Due to this tensor product structure, two-dimensional spectral filters can be applied in a separable fashion, one dimension at a time. To do this, one-dimensional slices of the 2D spectrum are selected and filtered with a 1D filter in each spatial direction independently. See \cref{fig:2d-spectrum} for an illustration of a two-dimensional Fourier spectrum and one-dimensional slices taken in each direction.

\begin{figure}[h!]
	\centering
	\begin{minipage}{.48\textwidth}
	\includegraphics[width=\textwidth]{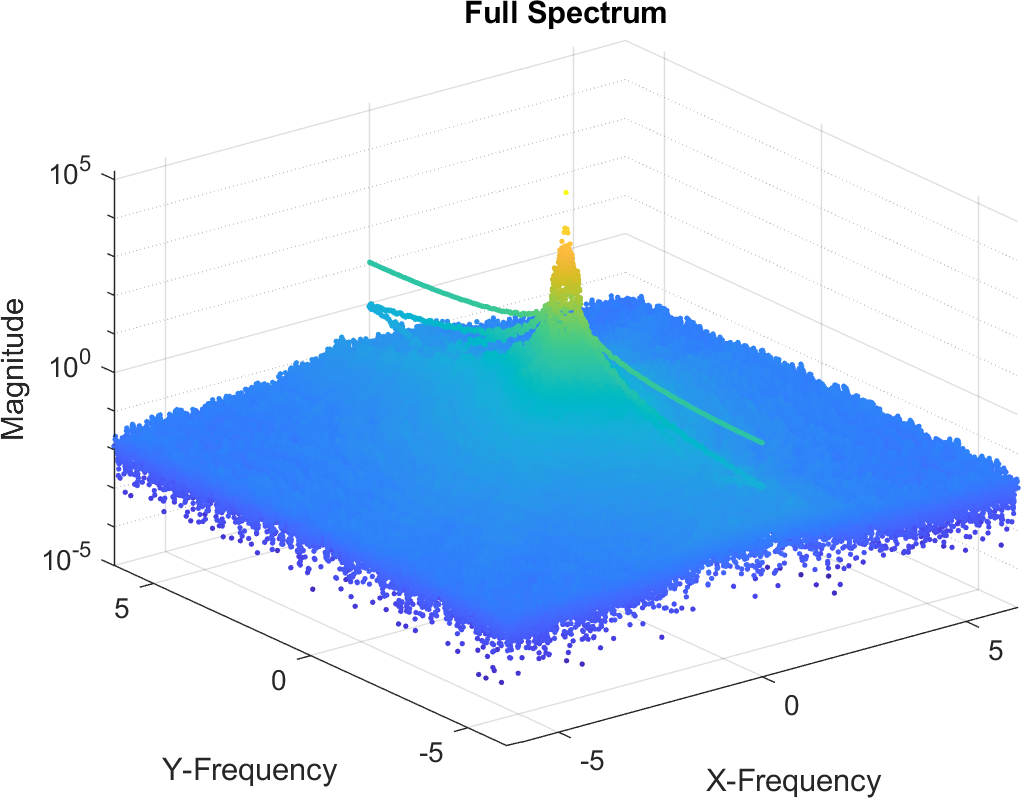} 
	\end{minipage}
	\begin{minipage}{.48\textwidth}
		\centering
		\includegraphics[width=\textwidth]{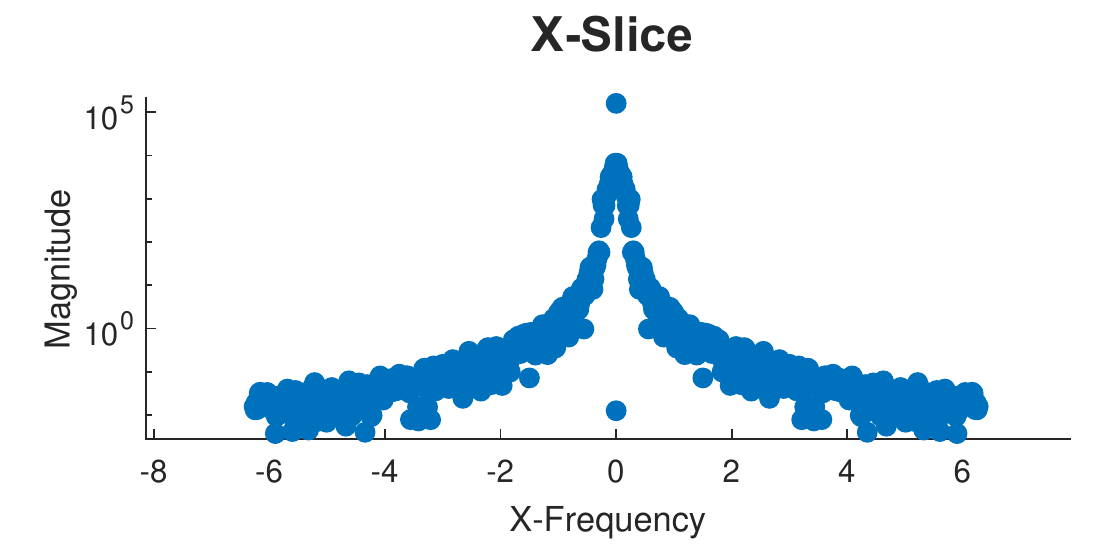}\\
		\includegraphics[width=\textwidth]{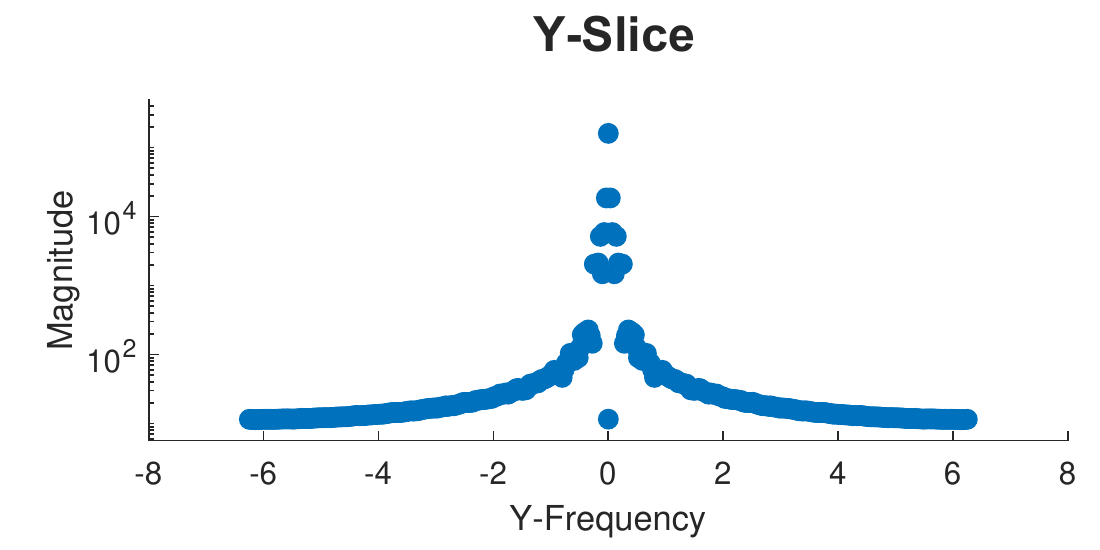}
	\end{minipage}
	\caption{Fourier spectrum of a two-dimensional function. At left: Full spectrum over frequency grid. At right: One dimensional slices of the Fourier modes taken in different dimensions.}
	\label{fig:2d-spectrum}
\end{figure}

\paragraph{Filters for Differentiation}
When modeling a two-dimensional surface with B-splines, it is necessary to consider partial derivatives in each dimension. To differentiate with respect to the first variable, we repeatedly apply the filter in \cref{eq:deriv-filter} to strands of coefficients with a fixed second index: $\{\hat{f}_{i, j}\}_{i=0}^{m_1-1}$. Once the filter operation has been applied for all $0 \leq j \leq m_2-1$, the partial derivatives are computed by taking the 2D IFFT of the filtered Fourier coefficients. This produces an approximation to $\frac{\partial^q f}{\partial x_1^q}$ at each grid point in the domain. Computing the second set of partial derivatives is done in the same way, except the filter operation is now applied to strands $\{\hat{f}_{i,j}\}_{j=1}^{m_2-1}$ where the first index is fixed.

Once both sets of partial derivatives have been computed, we construct two feature CDFs to guide the placement of knots---one for each dimension. As above, the process for each dimension is the same, up to a change in indexing. Therefore, we will only describe the process of creating a feature CDF for the first dimension. The feature function with respect to the first dimension is analogous to the one-dimensional case, with the first partial derivatives used in lieu of the full derivative. That is,
\begin{equation}
	F_1(u_1,u_2) = \abs{ \frac{\partial^q f}{\partial x_1^q}\left(a_1+u_1(b_1-a_1), a_2 + u_2(b_2-a_2)\right) }^{\frac{1}{q}}.
\end{equation}
Next, this two-dimensional feature function is collapsed along the second dimension to produce a new feature indicator that is a function of $u_1$ only. We set
\begin{equation}
	F_1^*(u_1) = \sum_{i=0}^{m_2-1} F_1\left( u_1, \frac{i}{m_2-1}\right).
\end{equation}
Thus $F_1^*$ will be large whenever $F_1(u_1,u_2)$ is large for some value of $u_2$, indicating a greater need for knots near $u_1$.

The feature CDF with respect to the first dimension is now simply the CDF of $F_1^*$:
\begin{equation}
	G_1(u) = \int_0^u F_1^*(t)\,dt.
\end{equation}
The knots in the first dimension are then chosen with $G_1$ according to the method described in \cref{ssec:deriv-method}. Then same procedure can be repeated to choose knots in the second dimension as well.

\paragraph{Filters for Smoothing} The smoothing filter in two dimensions is defined by a Gaussian in a manner similar to the one-dimensional case. In one dimension, the only problem dependent parameter in the smoothing filter was the mesh size $h$. On a two-dimensional equispaced grid, there are mesh sizes $h_1$ and $h_2$ describing the grid width for the first and second dimensions. The smoothing kernel takes on a form similar to \cref{eq:smooth-filter}, namely,
\begin{equation}
	K_s(\xi_1, \xi_2) = e^{-\frac{\pi^2}{2}(h_1^2\xi_1^2 + h_2^2\xi_2^2)}.
\end{equation}
To compute the smoothed version of a noisy 2D data set, it suffices to pointwise multiply $K_s$ with the Fourier coefficients and then take the 2D IFFT.

\paragraph{Filters for Jump Detection}
Like the filters for differentiation, applying filters to detect jumps in two-dimensional data involves a sense of directionality. At any point $(u_1,u_2)$ in parameter space, the signal may experience a jump in either cardinal direction. In fact, the signal may jump in a direction not aligned with either cardinal direction. Due to the tensor product structure of the grid points (and Fourier coefficients), however, we are constrained to only check for jumps along the cardinal directions. To create a jump detector with respect to the first dimension, the one-dimensional filter is multiplied with each strand $\{\hat{f}_{ij}\}_{i=0}^{m_1-1}$ like in the case of the differentiation filters. Applying the 2D IFFT to these filtered coefficients produces a two-dimensional jump detector function where only jumps with respect to the first coordinate are considered. The same process with coordinate indexing switched produces an equivalent detector for jumps along the second dimension.

\section{Numerical studies}\label{sec:numerical}
To test these techniques, we compared our approach with the method of Yeh et al.~\cite{yeh2020knot}. We refer to this procedure as ``derivative-informed (DI) knot placement.'' In its original presentation, the DI method was shown to have comparable or superior accuracy to a number of other knot placement schemes, both iterative and direct~\cite{yeh2020knot}. Here we intend to show that our method is consistent with the DI method for smoothly periodic signals. We then exhibit a number of cases that the DI method is not designed to handle and show that our method performs as expected. For reference, we also compare the accuracy of our method with uniform knot placement. These results are marked ``uniform'' in the following plots. For each test, we report the RMS and maximum error of the resulting spline. The behavior of the maximum error is especially important when considering nonsmooth data, since spurious overshoots and undershoots increase maximum error dramatically but may be unnoticed when reporting RMS error alone.

Several combinations of the techniques presented in \cref{sec:method} were tested. We have abbreviated these combinations as follows: 
\begin{itemize}
	\item[] \textbf{Derivative-Informed, Fourier (DI-F)}. Knot locations are determined by the DI method, but high-order derivatives are computed from the Fourier spectrum as described in \cref{ssec:deriv-knots}.
	\item[] \textbf{Derivative-Informed, Fourier+Smooth (DI-FS)}. Knots are placed using the derivatives of a smooth proxy function; these derivatives are computed using the procedure in \cref{ssec:smooth-proxy}. 
	\item[] \textbf{Derivative-Informed, Fourier+Jumps (DI-FJ)}. Jump discontinuities in the data and its derivative are identified as in \cref{ssec:jump-detect} and high-multiplicity knots are placed accordingly. Then, the remainder of the knots are placed according to the DI-FS method.
\end{itemize}

All methods were implemented in MATLAB, using built-in routines for the FFTs and evaluation of B-spline basis functions. Control points for the B-spline were determined via least squares minimization; in cases where the Schoenberg-Whitney condition is not satisfied (that is, when the least squares system is underdetermined), we selected the minimum norm solution. Least squares minimization was carried out with the complete orthogonal decomposition built-in to MATLAB.

\subsection{Noisy Data}
With the application of smoothing filters, our method can automatically handle data corrupted by noise. In general, the computation of numerical derivatives with finite differences is extremely unstable when applied to noisy data. For this reason, data sets with noise were not considered as a feasible use case in the original presentation of the DI method.  However, for the sake of comparison, we considered the performance of the DI method and uniform knot placement when assessing the accuracy of our method with noisy data.

To study the effect of noise on the method, we perturbed a smooth function with various amounts of noise. B-spline accuracy was studied as a function of total knots, for tests with different noise levels. In each case, Gaussian white noise with standard deviation equal to one was scaled by a constant factor. These factors were 0.0001, 0.001, 0.01, and 0.1. The results of these tests are summarized in \cref{fig:noise-test}. For each noise level, we see that the DI-FS method outperforms both the DI and uniform knot placement methods by an order of magnitude for small knot numbers. Given enough knots, all methods approach the same accuracy, which is the noise level for that particular test. With the application of smoothing filters, the DI-FS method is effectively placing knots according to the large scale features of the data set, which is the most effective way to increase accuracy with a small number of knots. As a result, the DI-FS method produces B-spline approximations accurate up to the level of noise with approximately half the knots needed by the other methods.

\begin{figure}[h!]
	\centering
	\includegraphics[width=.48\textwidth]{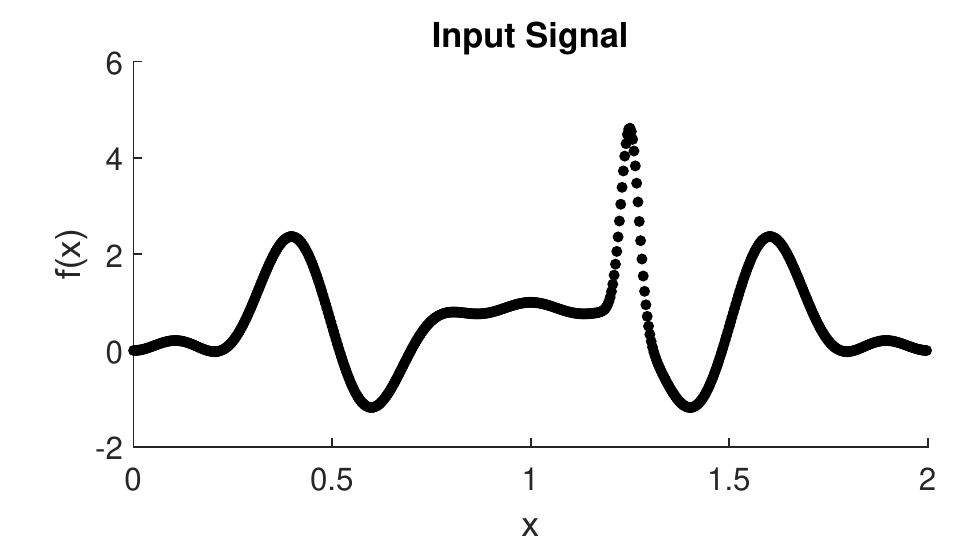} 
	\includegraphics[width=.48\textwidth]{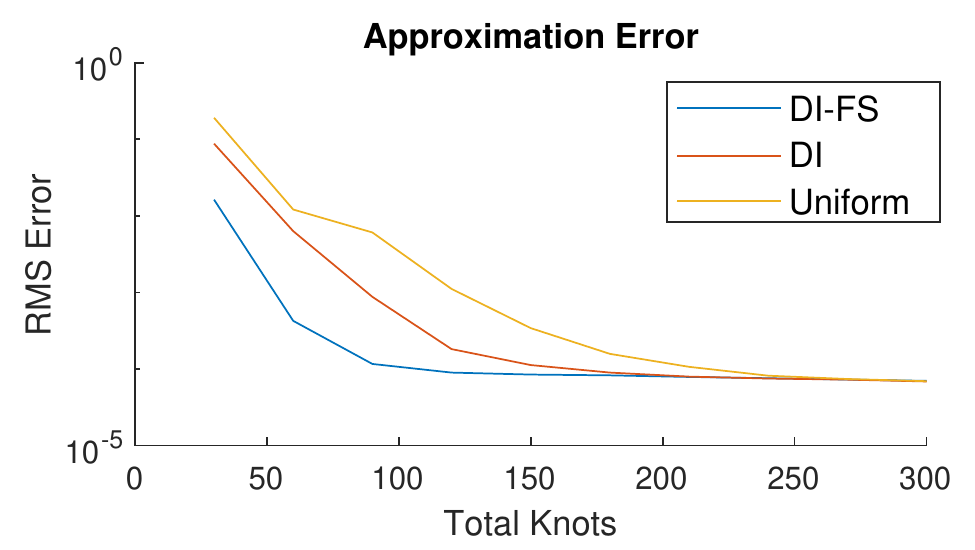} \\
	\includegraphics[width=.48\textwidth]{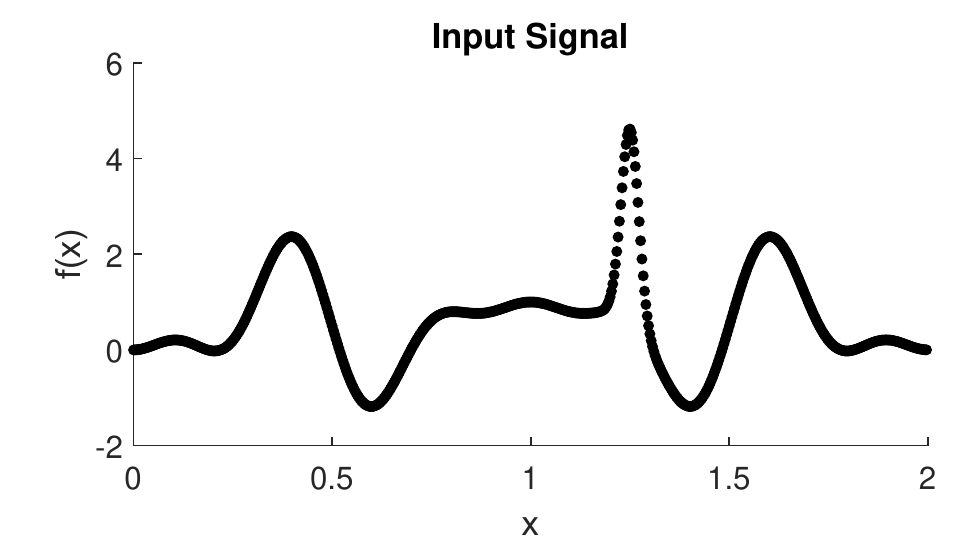} 
	\includegraphics[width=.48\textwidth]{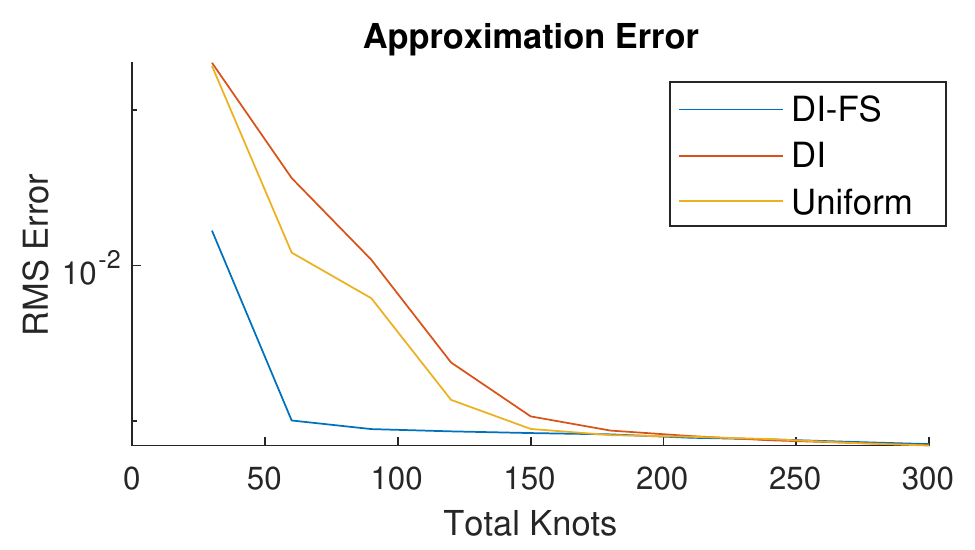} \\
	\includegraphics[width=.48\textwidth]{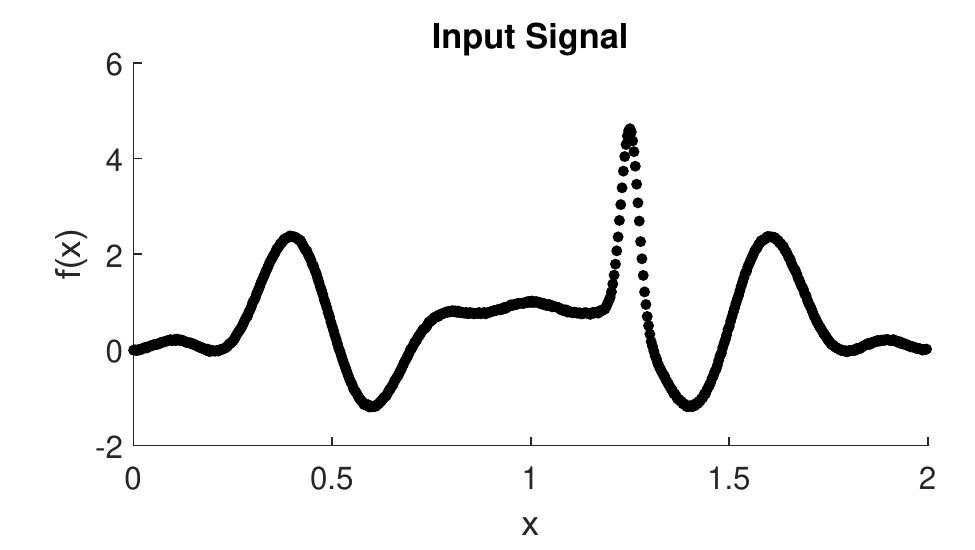}
	\includegraphics[width=.48\textwidth]{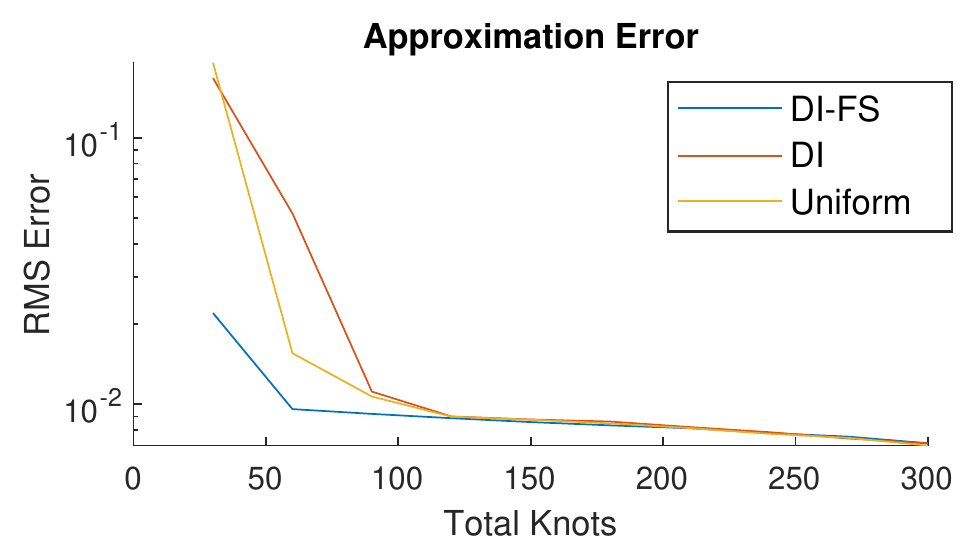} \\
	\includegraphics[width=.48\textwidth]{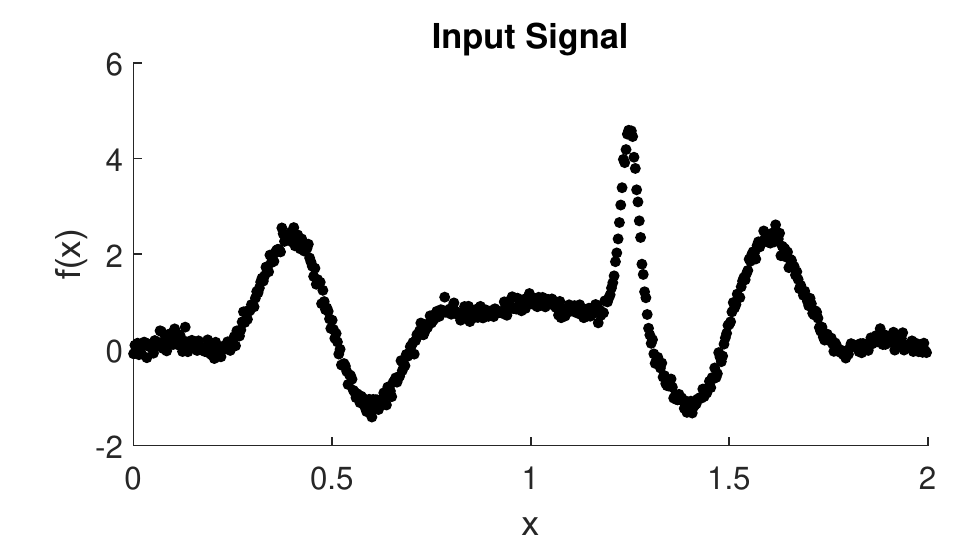} 
	\includegraphics[width=.48\textwidth]{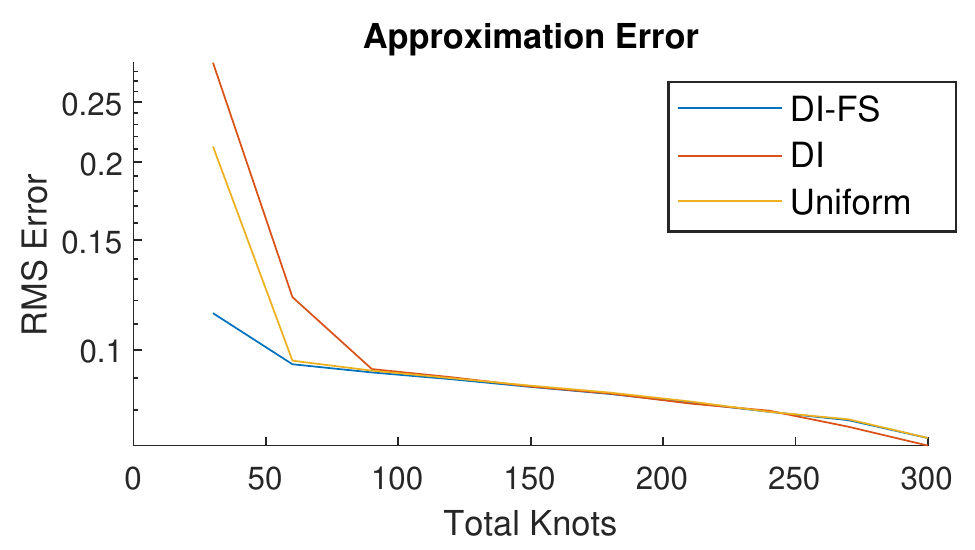}
	\caption{At left: Smooth data perturbed by increasing amounts of Gaussian white noise. Noise levels (from top to bottom) are: $10^{-4}$, $10^{-3}$, $10^{-2}$, $10^{-1}$. At right: Convergence of spline error as a function of total knots for each data set.}
	\label{fig:noise-test}
\end{figure}

\subsection{Piecewise-Smooth Data}
As we discussed in \cref{ssec:jump-detect}, attempting to fit a highly smooth spline to nonsmooth data is bound to produce an oscillatory approximation unless jumps in the data are treated carefully. To study the importance of jump detection for B-spline fitting, we tested the DI-FJ method on signals with a combination of $C^0$ and $C^1$ jumps and compared its performance with the DI method and uniform knot placement. This comparison contrasts the accuracy of using precisely-located high-multiplicity knots (our method) to the accuracy of using closely-packed simple knots (the DI method). In \cref{fig:test-jump}, we considered a signal with 600 input points containing one $C^1$ jump one-third of the way through the domain and one $C^0$ jump two-thirds of the way through the domain. For the purposes of the knot counts in \cref{fig:test-jump}, we consider one knot of multiplicity $q$ to be the same as $q$ simple knots. Using the two-pass strategy for detecting jumps, the DI-FJ pinpointed both jumps and placed high-multiplicity knots accordingly. Overall, we observe the DI-FJ method is more accurate than the DI method by an order of magnitude or more when the total number of knots is low. As we would expect from a method which corrects spurious oscillations, the increase in accuracy is especially significant in the maximum error norm. By targeting $C^0$ and $C^1$ jumps in the data, the DI-FJ method avoids the largest sources of error with a small amount of knots.
\begin{figure}
	\centering
	\includegraphics[width=.32\textwidth]{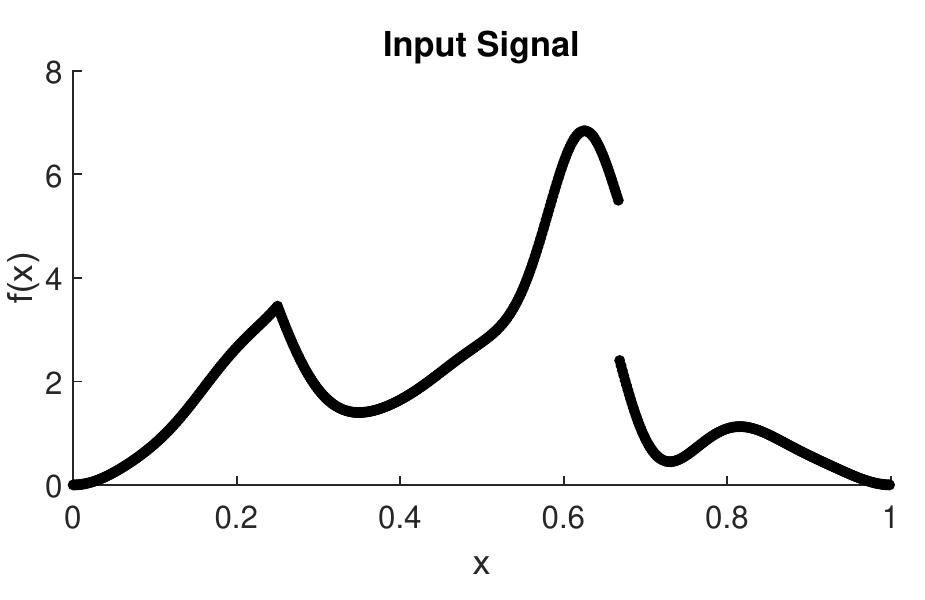}
	\includegraphics[width=.33\textwidth]{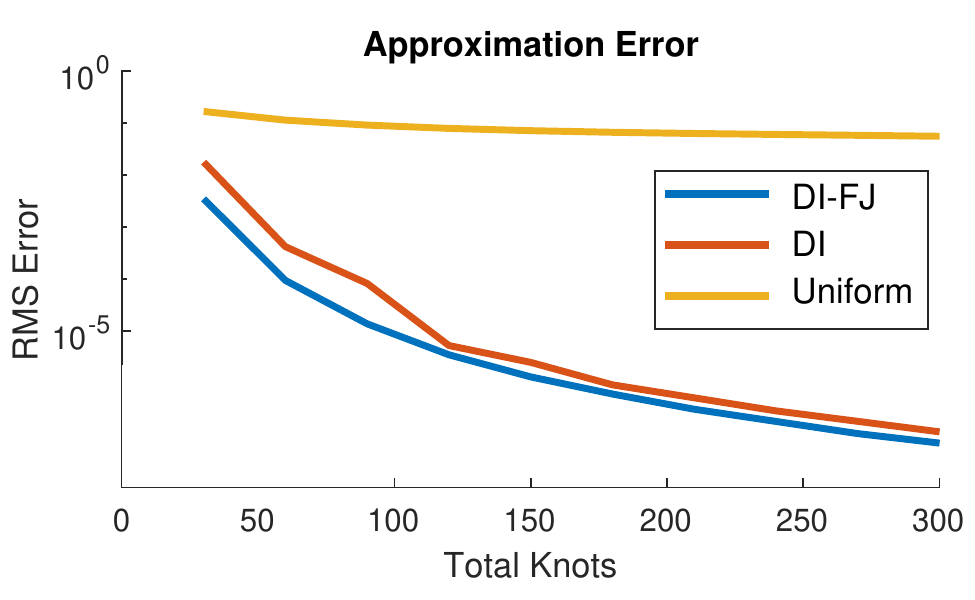}
	\includegraphics[width=.33\textwidth]{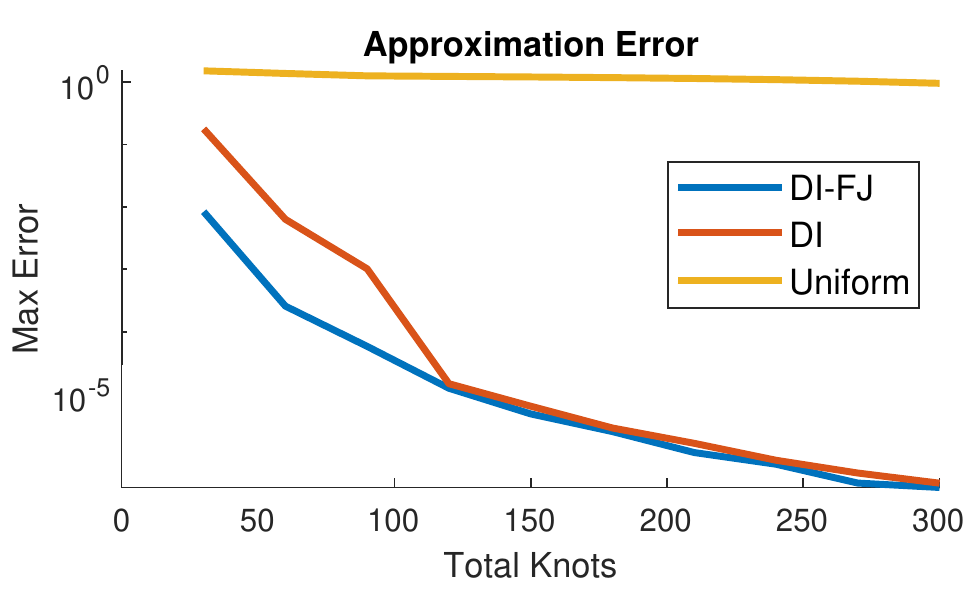}
	\vspace{-3ex}
	\caption{Comparison of the DI-FJ method to the DI method and uniform knot placement.}
	\label{fig:test-jump}
\end{figure}

\subsection{Two-Dimensional Data}
To validate the method in two dimensions, we tested our procedure on two data sets. 
\begin{itemize}
    \item \textbf{SPH Analytical:} An analytical, smooth signal generated from a combination of spherical harmonics modes ($Y^3_2 + Y^3_3$), which is doubly periodic.
    \item \textbf{Global Topography:} A spherical harmonic sampled expansion of composite satellite observation that integrates land topography and ocean bathymetry data over the entire Earth's surface, as provided by ETOPO1 Global Relief Model \cite{global-topo}
\end{itemize}

\subsubsection{Smooth Data}
When considering the first data set, we expect uniform knots will achieve near-optimal accuracy due to the highly smooth nature of the data. This is confirmed by our test: when placing knots according to the DI-F method, the resulting knot vector is close to uniform and the spline approximation error is almost identical to the splines constructed from uniform knots. \Cref{fig:2d-harmonic}  compares spline approximation error with derivative-informed and uniform knots when the approximating spline is fourth order (i.e., cubic).

\begin{figure}[h!]
	\centering
	\includegraphics[width=.3\textwidth]{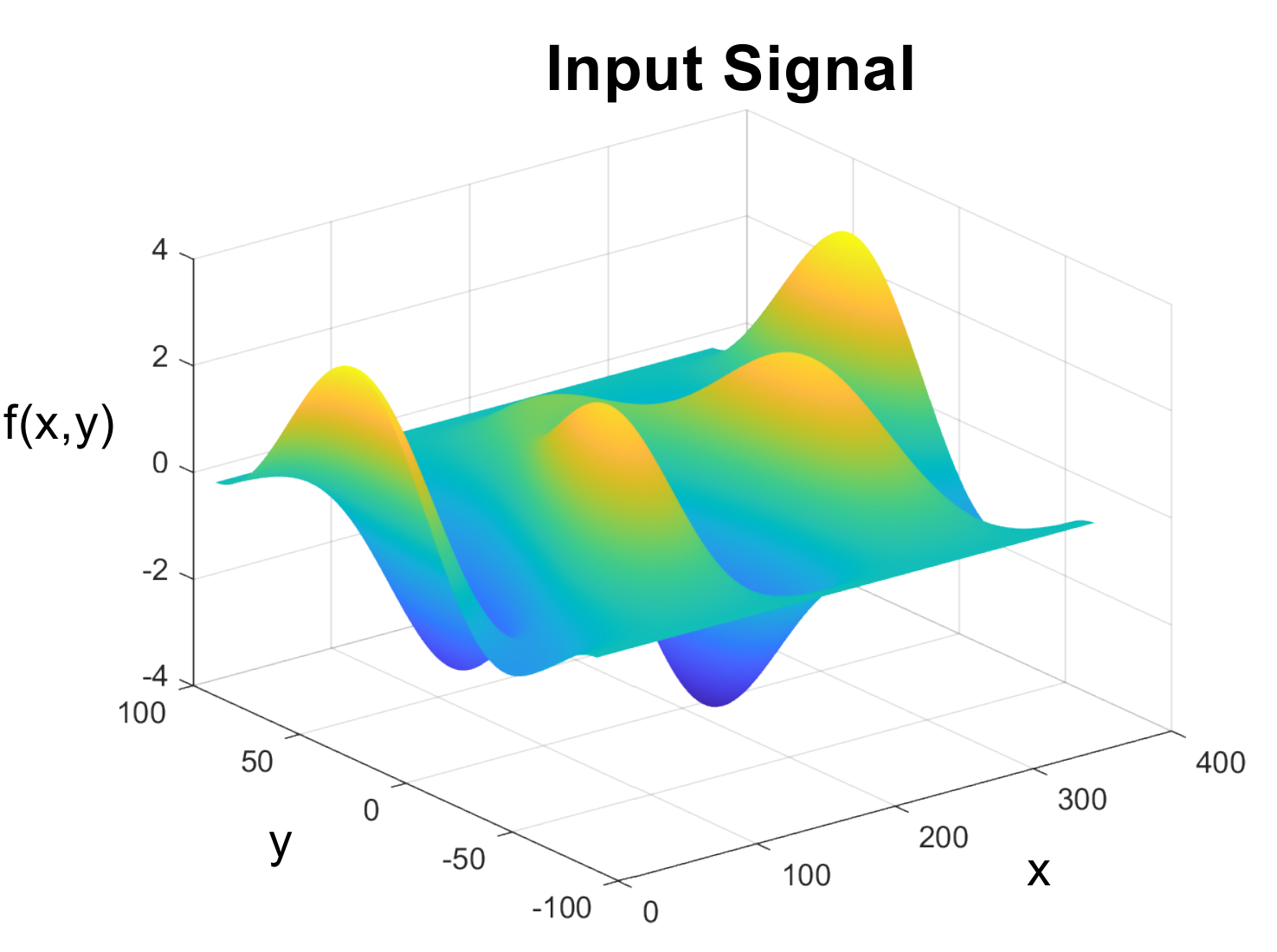} \hspace{.03\textwidth}
	\includegraphics[width=.3\textwidth]{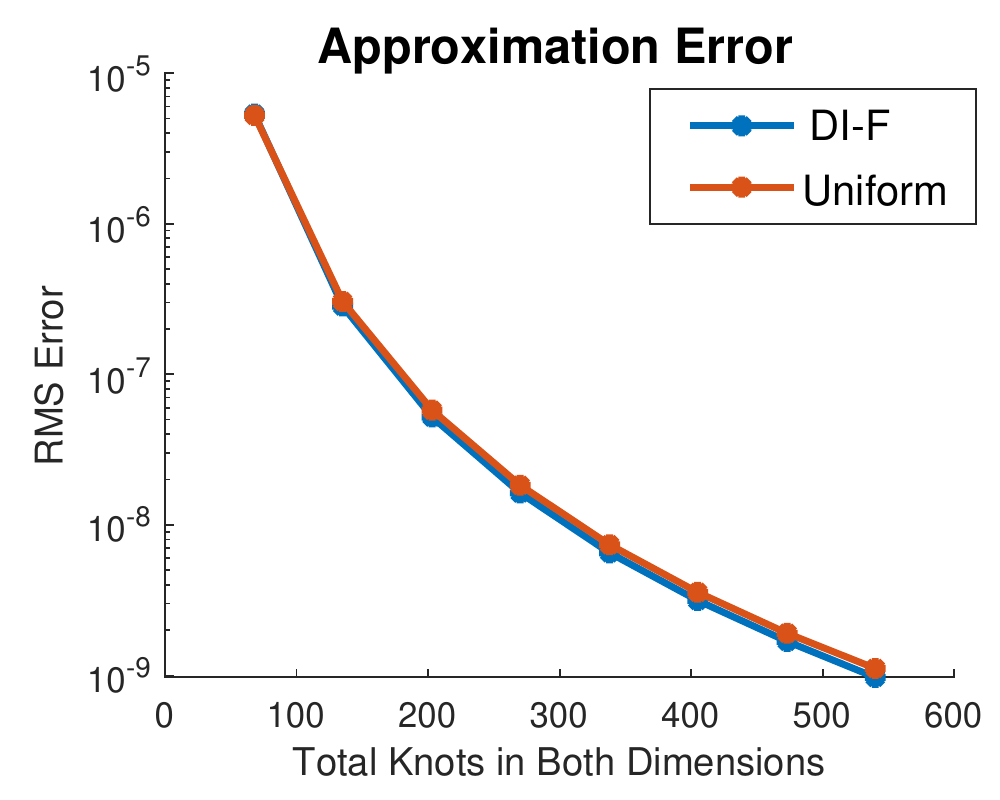} \hspace{.03\textwidth}
	\includegraphics[width=.3\textwidth]{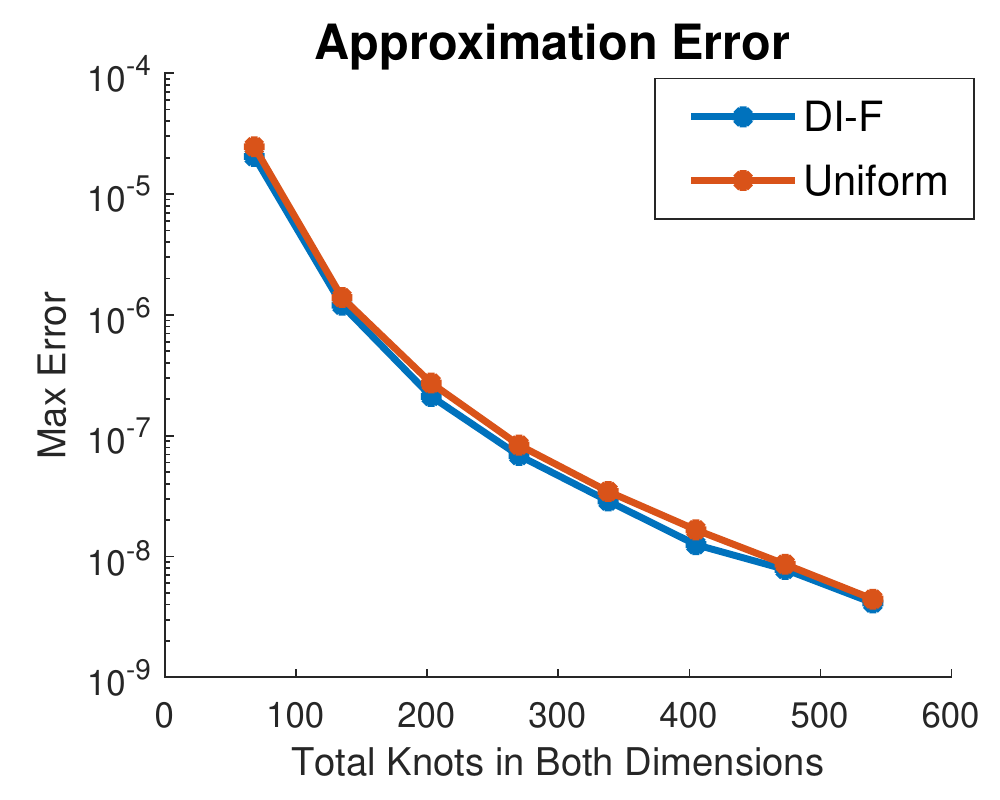}
	\caption{Comparison of spline approximation error as a function of knots for the SPH Analytical data set. Throughout the study, the ratio of knots in the first dimension to knots in the second dimension was kept constant.}
	\label{fig:2d-harmonic}
\end{figure}

\subsubsection{Non-smooth Data}
In the second test, the performance of the method was measured for the real-world field data set that represents the Earth's global topography. This data set was sampled on a uniform grid with a resolution of $720\times 360$, and retains periodicity only in the first dimension (X) and not in the second (Y). As a consequence, we computed the derivatives in the x-direction from the Fourier spectrum, while derivatives in the y-direction were computed using finite differences. Knot locations were then chosen in the manner described by \cref{ssec:2d-method}.

The results of this test are summarized in \cref{fig:2d-climate}. Here, we compare the error produced by a spline with uniformly spaced knots with the error produced by a spline with knots chosen by the DI-F method. The DI-F method outperforms uniform knot placement by approximately one order of magnitude for moderate knot counts. 

\begin{figure}[h!]
	\centering
	\includegraphics[width=.3\textwidth]{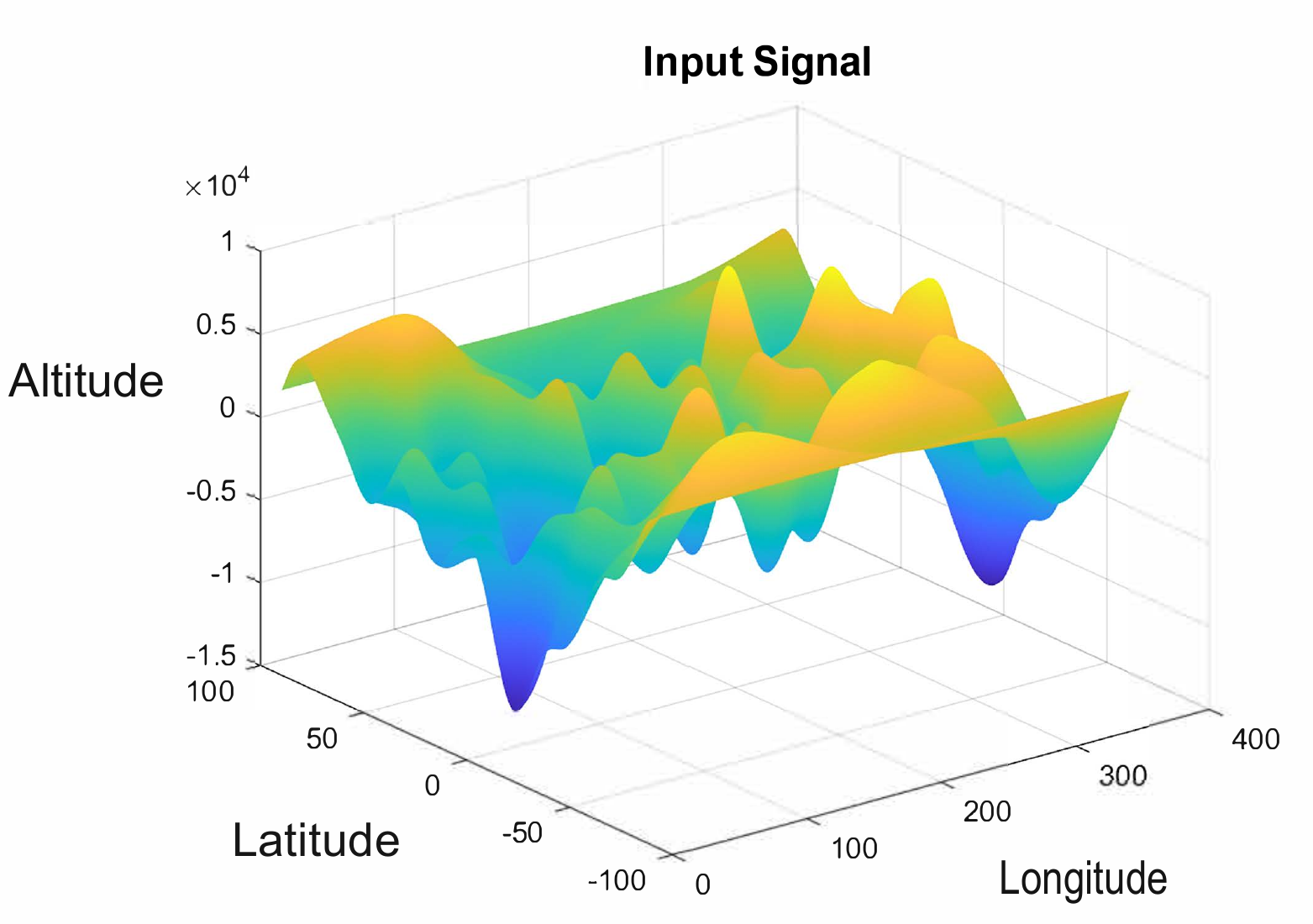}\hspace{.03\textwidth}
	\includegraphics[width=.3\textwidth]{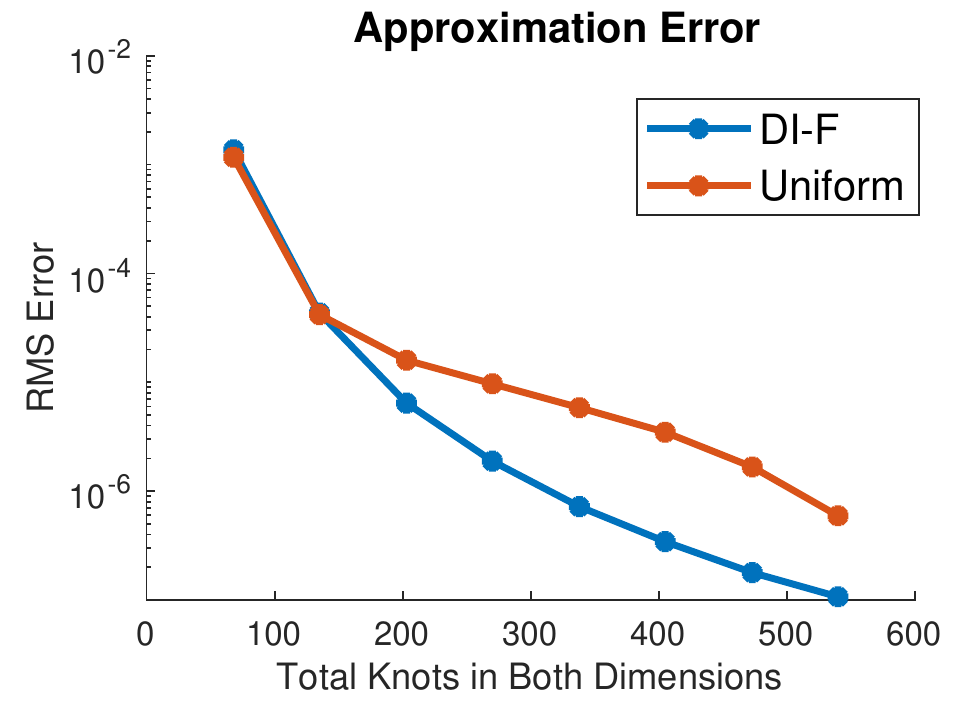}\hspace{.03\textwidth}
	\includegraphics[width=.3\textwidth]{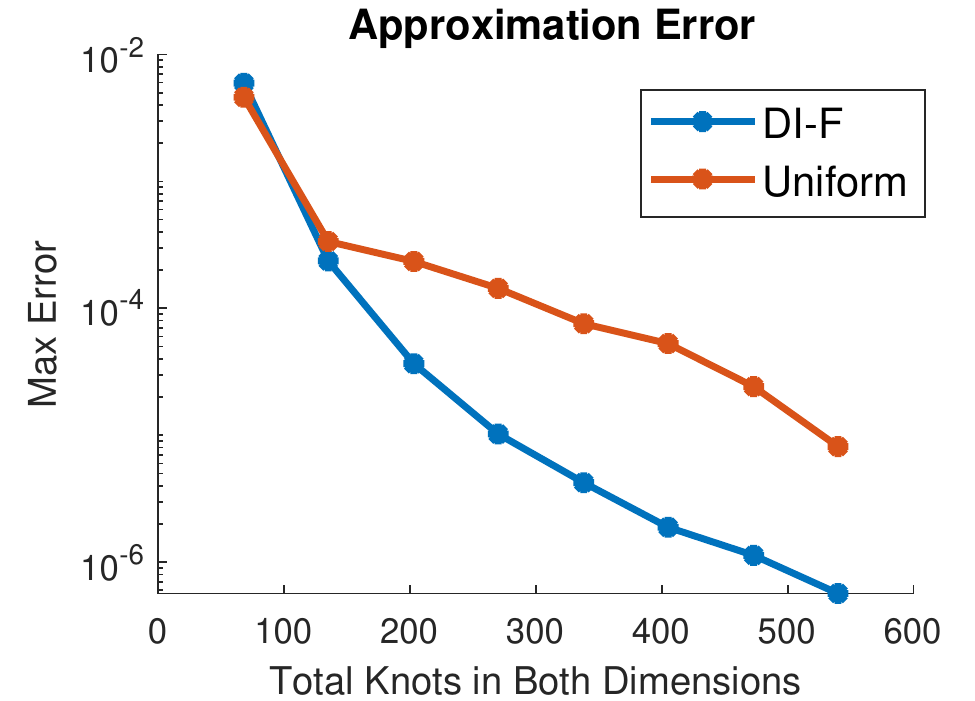}
	\caption{Comparison of spline approximation error as a function of knots for the Global Topography data set. Throughout the study, the ratio of knots in the first dimension to knots in the second dimension was kept constant.}
	\label{fig:2d-climate}
\end{figure}

\section{Future Work}\label{sec:future}
Having set the theoretical groundwork for a Fourier-informed approach to B-spline fitting, it is essential to note that the B-spline data models presented here need not be restricted to periodic signals. If the data are already represented on an equidistant grid but are not, then strategies for periodic continuation~\cite{albin_pathmanathan_2014, lyon2011continuation} should be considered. If the data are scattered, a naive approach would be to interpolate on an equidistant grid, which may add computational overhead to the FFT; however, the non-uniform FFT (NUFFT)~\cite{greengard2004} is recommended to preserve the efficiency of the FFT. If the data set is prohibitively large, it is is preferable to consider block transforms into spectral space, which can be achieved through the discrete Chebyshev transform (DCT)~\cite{boyd2001} or the discrete Legendre transform (DLT)~\cite{marin2016}. 
Once these extensions are studied, the methods described in this paper can be extended to more complex cases, e.g. large-scale scattered and nonperiodic data.

\section{Conclusions} \label{sec:conclusion}
We described a fast and scalable method to choose knot locations for B-spline approximations of periodic signals based on Fourier coefficients and spectral filters. Our method accounts for noise, jump discontinuities, and high-order derivatives of the signal when choosing knot locations. As a result, the method can place high-multiplicity knots where a signal contains jumps in its value or its derivative, eliminating Gibbs-like overshoots and undershoots near the jump. The method can also compute approximate derivatives in the presence of noise. As a result, our knot placement algorithm efficiently uses the knots available to it and is more accurate than a competing method when B-splines are constructed from a small number of knots.

In addition to performing well in several challenging settings, our knot placement scheme is designed to be fast and flexible. All computations necessary to place knots are carried out in Fourier space by filtering the Fourier spectrum of the input in various ways. This allows different filters to be combined easily and makes the method extensible to additional filters in the future. Furthermore, the computational complexity of applying these filtering operations is $O(N\log N)$ and highly scalable, making the implementation suitable to tackle large data sets often encountered in scientific applications.

\section*{Acknowledgments}
This work is supported by the
U.S. Department of Energy, Office of Science, Advanced Scientific Computing Research under Contract DE-AC02-06CH11357, and the Exascale
Computing Project (Contract No. 17-SC-20-SC), a collaborative effort of the U.S. Department of Energy Office of Science and the National Nuclear Security Administration.

\bibliographystyle{siamplain}
\bibliography{mfa}

\begin{thebibliography}{10}

\bibitem{glvis-tool}
{\em {GLVis}: Opengl finite element visualization tool}.
\newblock glvis.org, \url{https://doi.org/10.11578/dc.20171025.1249}.

\bibitem{albin_pathmanathan_2014}
{\sc N.~Albin and S.~Pathmanathan}, {\em Discrete periodic extension using an
  approximate step function}, SIAM Journal on Scientific Computing, 36 (2014),
  pp.~A668--A692, \url{https://doi.org/10.1137/130932533}.

\bibitem{boyd2001}
{\sc J.~P. Boyd}, {\em Chebyshev and Fourier spectral methods}, Courier
  Corporation, 2001.

\bibitem{cates2007detecting}
{\sc D.~Cates and A.~Gelb}, {\em Detecting derivative discontinuity locations
  in piecewise continuous functions from fourier spectral data}, Numerical
  Algorithms, 46 (2007), pp.~59--84,
  \url{https://doi.org/10.1007/s11075-007-9127-x}.

\bibitem{conti2001cubic}
{\sc C.~Conti, R.~Morandi, C.~Rabut, and A.~Sestini}, {\em Cubic spline data
  reduction choosing the knots from a third derivative criterion}, Numerical
  Algorithms, 28 (2001), pp.~45--61,
  \url{https://doi.org/10.1023/A:1014022210828}.

\bibitem{cox1972numerical}
{\sc M.~G. Cox}, {\em The numerical evaluation of b-splines}, IMA Journal of
  Applied Mathematics, 10 (1972), pp.~134--149,
  \url{https://doi.org/10.1093/imamat/10.2.134}.

\bibitem{de1972calculating}
{\sc C.~{de Boor}}, {\em On calculating with b-splines}, Journal of
  Approximation Theory, 6 (1972), pp.~50--62,
  \url{https://doi.org/10.1016/0021-9045(72)90080-9}.

\bibitem{gelb2008detection}
{\sc A.~Gelb and D.~Cates}, {\em Detection of edges in spectral data iii:
  refinement of the concentration method}, Journal of Scientific Computing, 36
  (2008), pp.~1--43, \url{https://doi.org/10.1007/s10915-007-9170-8}.

\bibitem{gelb1999detection}
{\sc A.~Gelb and E.~Tadmor}, {\em Detection of edges in spectral data}, Applied
  and computational harmonic analysis, 7 (1999), pp.~101--135,
  \url{https://doi.org/10.1006/acha.1999.0262}.

\bibitem{gelb2000detection}
{\sc A.~Gelb and E.~Tadmor}, {\em Detection of edges in spectral data ii:
  Nonlinear enhancement}, SIAM Journal on Numerical Analysis, 38 (2000),
  pp.~1389--1408, \url{https://doi.org/10.1137/S0036142999359153}.

\bibitem{greengard2004}
{\sc L.~Greengard and J.-Y. Lee}, {\em Accelerating the nonuniform fast fourier
  transform}, SIAM review, 46 (2004), pp.~443--454,
  \url{https://doi.org/10.1137/S003614450343200X}.

\bibitem{galvez2015}
{\sc A.~Gálvez, A.~Iglesias, A.~Avila, C.~Otero, R.~Arias, and C.~Manchado},
  {\em Elitist clonal selection algorithm for optimal choice of free knots in
  b-spline data fitting}, Applied Soft Computing, 26 (2015), pp.~90 -- 106,
  \url{https://doi.org/10.1016/j.asoc.2014.09.030}.

\bibitem{habermann2007multi}
{\sc C.~Habermann and F.~Kindermann}, {\em Multidimensional spline
  interpolation: Theory and applications}, Computational Economics, 30 (2007),
  pp.~153--169, \url{https://doi.org/10.1007/s10614-007-9092-4}.

\bibitem{hughes2005iga}
{\sc T.~Hughes, J.~Cottrell, and Y.~Bazilevs}, {\em Isogeometric analysis: Cad,
  finite elements, nurbs, exact geometry and mesh refinement}, Computer Methods
  in Applied Mechanics and Engineering, 194 (2005), pp.~4135 -- 4195,
  \url{https://doi.org/10.1016/j.cma.2004.10.008}.

\bibitem{jupp1978free}
{\sc D.~L.~B. Jupp}, {\em Approximation to data by splines with free knots},
  SIAM Journal on Numerical Analysis, 15 (1978), pp.~328--343,
  \url{https://doi.org/10.1137/0715022}.

\bibitem{movinglsq}
{\sc D.~Levin}, {\em The approximation power of moving least-squares},
  Mathematics of Computation, 67 (1998), pp.~1517--1531,
  \url{https://doi.org/10.1090/S0025-5718-98-00974-0}.

\bibitem{li2004}
{\sc W.~Li, S.~Xu, G.~Zhao, and L.~P. Goh}, {\em A heuristic knot placement
  algorithm for b-spline curve approximation}, Computer-Aided Design and
  Applications, 1 (2004), pp.~727--732,
  \url{https://doi.org/10.1080/16864360.2004.10738319}.

\bibitem{liang2017}
{\sc F.~Liang, J.~Zhao, S.~Ji, C.~Fan, and B.~Zhang}, {\em A novel knot
  selection method for the error-bounded b-spline curve fitting of sampling
  points in the measuring process}, Measurement Science and Technology, 28
  (2017), p.~065015, \url{https://doi.org/10.1088/1361-6501/aa6a05}.

\bibitem{lin2018geometric}
{\sc H.~Lin, T.~Maekawa, and C.~Deng}, {\em Survey on geometric iterative
  methods and their applications}, Computer-Aided Design, 95 (2018), pp.~40 --
  51, \url{https://doi.org/10.1016/j.cad.2017.10.002}.

\bibitem{lyon2011continuation}
{\sc M.~Lyon}, {\em A fast algorithm for fourier continuation}, SIAM Journal on
  Scientific Computing, 33 (2011), pp.~3241--3260,
  \url{https://doi.org/10.1137/11082436X}.

\bibitem{marin2016}
{\sc O.~Marin, M.~Schanen, and P.~Fischer}, {\em Large-scale lossy data
  compression based on an a priori error estimator in a spectral element code},
  tech. report, ANL/MCS-P6024-0616, 2016,
  \url{https://www.mcs.anl.gov/papers/P6024-0616.pdf} (accessed 2020-10-06).

\bibitem{michel2020}
{\sc D.~Michel and A.~Zidna}, {\em A new deterministic heuristic knots
  placement for b-spline approximation}, Mathematics and Computers in
  Simulation,  (2020), \url{https://doi.org/10.1016/j.matcom.2020.07.021}.

\bibitem{nashed2019}
{\sc Y.~S. Nashed, T.~Peterka, V.~Mahadevan, and I.~Grindeanu}, {\em Rational
  approximation of scientific data}, in Computational Science -- ICCS 2019,
  J.~M.~F. Rodrigues, P.~J.~S. Cardoso, J.~Monteiro, R.~Lam, V.~V.
  Krzhizhanovskaya, M.~H. Lees, J.~J. Dongarra, and P.~M. Sloot, eds.,
  Springer, 2019, pp.~18--31,
  \url{https://doi.org/10.1007/978-3-030-22734-0_2}.

\bibitem{global-topo}
{\sc N.~Oceanic and A.~Administration}, {\em Etopo1 global relief model},
  \url{https://www.ngdc.noaa.gov/mgg/global/global.html} (accessed 2020/10/02).

\bibitem{peterka_ldav18}
{\sc T.~Peterka, Y.~Nashed, I.~Grindeanu, V.~Mahadevan, R.~Yeh, and
  X.~Trixoche}, {\em {Foundations of Multivariate Functional Approximation for
  Scientific Data}}, in Proceedings of 2018 IEEE Symposium on Large Data
  Analysis and Visualization, 2018,
  \url{https://doi.org/10.1109/LDAV.2018.8739195}.

\bibitem{richards1991gibbs}
{\sc F.~Richards}, {\em A gibbs phenomenon for spline functions}, Journal of
  approximation theory, 66 (1991), pp.~334--351,
  \url{https://doi.org/10.1016/0021-9045(91)90034-8}.

\bibitem{rhino3d}
{\sc {Robert McNeel \& Associates}}, {\em {Rhinoceros}}, \url{rhino3d.com}
  (accessed 2020/10/01).
\newblock Version 6.

\bibitem{basic_theory}
{\sc L.~Schumaker}, {\em Spline Functions: Basic Theory}, Cambridge
  Mathematical Library, Cambridge University Press, 3~ed., 2007,
  \url{https://doi.org/10.1017/CBO9780511618994}.

\bibitem{tadmor2007filters}
{\sc E.~Tadmor}, {\em Filters, mollifiers and the computation of the gibbs
  phenomenon}, Acta Numerica, 16 (2007), pp.~305--378,
  \url{https://doi.org/10.1017/S0962492906320016}.

\bibitem{yeh2020knot}
{\sc R.~Yeh, Y.~S.~G. Nashed, T.~Peterka, and X.~Tricoche}, {\em Fast automatic
  knot placement method for accurate b-spline curve fitting}, Computer-Aided
  Design, 128 (2020), 102905, \url{https://doi.org/10.1016/j.cad.2020.102905}.

\end{thebibliography}

\pagebreak
The submitted manuscript has been created by UChicago Argonne, LLC, Operator of Argonne National Laboratory ("Argonne”). Argonne, a U.S. Department of Energy Office of Science laboratory, is operated under Contract No. DE-AC02-06CH11357. The U.S. Government retains for itself, and others acting on its behalf, a paid-up nonexclusive, irrevocable worldwide license in said article to reproduce, prepare derivative works, distribute copies to the public, and perform publicly and display publicly, by or on behalf of the Government. The Department of Energy will provide public access to these results of federally sponsored research in accordance with the DOE Public Access Plan (\url{http://energy.gov/downloads/doe-public-access-plan}).

\end{document}